\documentclass[12pt]{amsart}
\usepackage{amsmath,amssymb,graphicx,mathrsfs}
\usepackage{stmaryrd,amsthm,CJK,color}
\SetSymbolFont{stmry}{bold}{U}{stmry}{m}{n}
\usepackage{times}
\usepackage[colorlinks=true,citecolor = blue]{hyperref}

\usepackage{amsmath,bm}
\numberwithin{equation}{section}
\numberwithin{equation}{section}
\newtheorem{thm}{Theorem}[section]
\newtheorem{prop}[thm]{Proposition}
\newtheorem{lem}[thm]{Lemma}
\newtheorem{rem}[thm]{Remark}
\newtheorem{cor}[thm]{Corollary}

\renewcommand{\theequation}{\thesection.\arabic{equation}}

\newcommand{\beq}{\begin{equation}}
\newcommand{\eeq}{\end{equation}}

\newcommand{\C}{{\mathbb C}}
\newcommand{\Z}{{\mathbb Z}}

\newcommand{\cH}{\mathcal{H}}

\newcommand{\cP}{\mathcal{P}}

\newcommand{\bs}{\boldsymbol}

\newcommand{\g}{\mathfrak{g}}
\newcommand{\h}{\mathfrak{h}}
\newcommand{\n}{\mathfrak{n}}

\newcommand{\ve}{\varepsilon}

\newcommand{\la}{\lambda}
\newcommand{\La}{\Lambda}

\newcommand{\gge}{\geqslant}
\newcommand{\lle}{\leqslant}

\newcommand{\mc}{\mathcal}
\newcommand{\tl}{\tilde}

\newcommand{\pa}{{\partial}}
\makeatletter
\def\@eqnnum{{\normalfont \color{red} (\theequation)}}
\makeatother

\begin{document}
\pagestyle{myheadings}

\setcounter{page}{1}

\title{On the Gaudin model associated to Lie algebras of classical types}

\author{Kang Lu, E. Mukhin, and A. Varchenko}
\address{K.L.: Department of Mathematical Sciences,
Indiana University-Purdue University-Indianapolis,
402 N.Blackford St., LD 270,
Indianapolis, IN 46202, USA}\email{lukang@iupui.edu}
\address{E.M.: Department of Mathematical Sciences,
Indiana University-Purdue University-Indianapolis,
402 N.Blackford St., LD 270,
Indianapolis, IN 46202, USA}\email{mukhin@math.iupui.edu}
\address{A.V.: Department of Mathematics, University of North Carolina at Chapel Hill,
Chapel Hill, NC 27599-3250, USA}
\email{anv@email.unc.edu}

\begin{abstract}
We derive explicit formulas for solutions of the Bethe ansatz equations of the Gaudin model associated to the tensor product of one arbitrary finite-dimensional irreducible module and one vector representation for all simple Lie algebras of classical type. We use this result to show that the Bethe ansatz is complete in any tensor product where all but one factor are vector representations and the evaluation parameters are generic. 
\end{abstract}

\maketitle

\section{Introduction}
The Gaudin Hamiltonians are an important example of a family of commuting operators. We study the case when the Gaudin Hamiltonians possess a symmetry given by the diagonal action of $\g$. In this case the Gaudin Hamiltonians depend on a choice of a simple Lie algebra $\g$, $\g$-modules $V_1,\dots,V_n$ and distinct complex numbers $z_1,\dots,z_n$, see \eqref{eq hamil}.

The problem of studying the spectrum of the Gaudin Hamiltonians has received a lot of attention.
However, the majority of the work has been done in type A. In this paper we study the cases of types B, C and D.

The main approach is the Bethe ansatz method. Our goal is to establish the method when all but one modules $V_i$ are isomorphic to the first fundamental representation $V_{\omega_1}$. Namely, we show that the Bethe ansatz equations have sufficiently many solutions and that the Bethe vectors constructed from those solutions form a basis in the space of singular vectors of $V_1\otimes\dots\otimes V_n$.

\medskip

The solution of a similar problem in type A in \cite{MV1} led to several important results, such as a proof of the strong form of the Shapiro-Shapiro conjecture for Grassmanians, simplicity of the spectrum of higher Gauding Hamiltonians, the bijection betweem Fuchsian differential operators without monodromy with the Bethe vectors, etc, see \cite{MTV2} and references therein. We hope that this paper will give a start to similar studies in type B.  In addition, the explicit formulas for simplest examples outside type A are important as
experimental data for testing various conjectures.

\medskip

By the standard methods, the problem is reduced to the case of $n=2$, with $V_1$ being an arbitrary finite-dimensional module, $V_2=V_{\omega_1}$  and $z_1=0$, $z_2=1$.
The reduction involves taking appropriate limits, when all points $z_i$ go to the same number with different rates. Then the $n=2$ problems are observed in
the leading order and the generic situation is recovered from the limiting case by the usual argument of deformations of isolated solutions of algebraic systems, see \cite{MV1} and Section \ref{sec B generic} for details.

For the 2-point case when one of the modules is the defining representation $V_{\omega_1}$, the spaces of singular vectors of a given weight are either trivial or one-dimensional. Then, according to the general philosophy, see \cite{MV3}, one would expect to solve the Bethe ansatz equations explicitly. In type A
it was done in \cite{MV3}. In the supersymmetric case of $\mathfrak{gl}(p|q)$ the corresponding Bethe ansatz equations are solved in \cite{MVY}. The other known cases with one dimensional spaces include tensor products of two arbitrary irreducible $\mathfrak{sl}_2$ modules, see \cite{V2} and tensor products of an arbitrary module with a symmetric power $V_{k\omega_1}$ of the vector representation in the case of $\mathfrak{sl}_{r+1}$, see \cite{MV5}. Interestingly, in the latter case the solutions of the Bethe ansatz equations are related to zeros of Jacobi-Pineiro polynomials which are multiple orthogonal polynomials.

In all previously known cases when the dimension of the space of singular vectors of a given weight is one, the elementary symmetric functions of solutions of Bethe ansatz equations completely factorize into products of linear functions of the parameters. This was one of the main reasons the formulas were found essentially by brute force. However, unexpectedly, the computer experiments showed that in types B, C, D, the formulas do not factorize, see also Theorem 5.5 in \cite{MV1}, and therefore, the problem remained unsolved. In this paper we present a method to compute the answer systematically.

\medskip

Our idea comes from the reproduction procedure studied in \cite{MV4}.
Let $V_1=V_\lambda$ be the irreducible module of highest weight $\lambda$, let $V_2,\dots,V_n$ be finite-dimensional irreducible modules, and let $l_1,\dots,l_r$ be nonnegative integers, where $r$ is the rank of $\g$. Fix distinct complex numbers $z_1=0,z_1, \dots, z_n$. Consider the Bethe ansatz equation, see \eqref{eq:bae},
associated to these data. Set $V=V_2\otimes\dots\otimes V_n$, denote the highest weight vector of $V$ by  $v^+$, the weight of $v^+$ by $\mu^+$, and set $\mu=\mu^+-\sum_{i=1}^r l_i\alpha_i$. Here $\alpha_i$ are simple roots of $\g$.

Given an isolated solution of the Bethe ansatz equations we can produce two Bethe vectors: one in the space of singular vectors in $V_\la\otimes V$ of weight $\mu+\lambda$ and another one in the space of vectors in $V$ of weight $\mu$. The first Bethe vector, see \eqref{wtf}, is an eigenvector of the standard Gaudin Hamiltonians, see \eqref{eq hamil}, acting in $V_\lambda\otimes V$ and the second Bethe vector is an eigenvector of trigonometric Gaudin Hamiltonians, see \cite{MV4}.
The second vector is a projection of the first vector to the space $v^+\otimes V\simeq V$.

Then the reproduction procedure of \cite{MV4} in the $j$-th direction allows us to construct a new solution of the Bethe ansatz equation associated to new data: representations $V_1=V_{s_j\cdot\la}$,  $V_2,\dots,V_n$ and integers $l_1,\dots,\tilde l_j,\dots,l_r$ so that the new weight
 $\tilde \mu=\mu^+-\sum_{i\neq j} l_i\alpha_i-\tilde l_j\alpha_j$ is given by $\tilde\mu=s_j\mu$.
This construction is quite general, it works for all symmetrizable Kac-Moody algebras provided that the weight $\lambda$ is generic, see Theorem \ref{thm:2.5} below. It gives a bijection between solutions corresponding to weights $\mu$ of $V$ in the same Weyl orbit.

Note that in the case $\mu=\mu^+$, the Bethe ansatz equations are trivial. Therefore, using the trivial solution and the reproduction procedure, we, in principal, can obtain solutions for all weights of the form: $\mu=w\mu^+$. Note also that in the case of the vector representation, $V=V_{\omega_1}$, all weights in $V$ are in the Weyl orbit of $\mu^+=\omega_1$ (with the exception of weight $\mu=0$ in type B). Therefore, we get all the solutions we need that way (the exceptional weight is easy to treat separately).

\medskip

In contrast to \cite{MV4}, we do not have the luxury of generic weight $\lambda$, and we have to check some technical conditions on each reproduction step. It turns out, such checks are easy when going to the trivial solution, but not the other way, see Section \ref{sec lemmas}. We manage to solve the recursion and obtain explicit formulas, see Corollary \ref{thm sol B} for type B, Theorem \ref{thm sol C} for type C and Theorem \ref{thm sol D} for type D. We complete the check using these formulas, see Section \ref{sec check generic}.

\medskip

To each solution of Bethe ansatz, one can associate an oper. For types A, B, C the oper becomes a scalar differential operator with rational coefficients, see \cite{MV2}, and Sections \ref{sec oper}, \ref{C sec}. In fact, the coefficients of this operator are eigenvalues of higher Gaudin Hamiltonians, see \cite{MTV} for type A and \cite{MM} for types B, C. The differential operators for the solutions obtained via the reproduction procedure are closely related. It allows us to give simple formulas for the differential operators related to our solutions, see Propositions \ref{diff B} and \ref{diff C}.  According to \cite{MV2}, the kernel of the differential operator is a space of polynomials with a symmetry, called a self-dual space.
We intend to discuss the self-dual spaces related to our situation in detail elsewhere.

\medskip

The paper is constructed as follows.  In Section  \ref{sec start} we describe the problem and set our notation. We study in detail the case of type B in Sections \ref{sec n=2, B} and \ref{sec B generic}.
In Section \ref{sec n=2, B} we solve the Bethe ansatz equation for $n=2$ when one of the modules is $V_{\omega_1}$. In Section \ref{sec B generic}, we use the results of Section \ref{sec n=2, B}  to show the completeness and simplicity of the spectrum of Gaudin Hamiltonians acting in tensor products where all but one factors are $V_{\omega_1}$, for generic values of $z_i$. In Section \ref{sec C and D} we give the corresponding formulas and statements in types C and D.

\bigskip

{\bf Acknowledgments.} 
This work was partially supported by a grant from the Simons Foundation (\#336826 to Alexander Varchenko and \#353831 to Evgeny Mukhin)."
The research of A.V. is supported by NSF grant DMS-1362924. A.V. thanks the MPI in Bonn for hospitality during his visit.

\section{The Gaudin model and Bethe ansatz}\label{sec start}
\subsection{Simple Lie algebras}
Let $\g$ be a simple Lie algebra over $\C$ with Cartan matrix $A=(a_{i,j})_{i,j=1}^r$. Denote the universal enveloping algebra of $\g$ by $\mc U(\g)$. Let $D=\mathrm{diag}\{d_1,\dots,d_r\}$ be the diagonal matrix with positive relatively prime integers $d_i$ such that $B=DA$ is symmetric.

Let $\h\subset\g$ be the Cartan subalgebra. Fix simple roots $\alpha_1,\dots,\alpha_r$ in $\h^*$. Let $\alpha_1^{\vee},\dots,\alpha_r^{\vee}\in \h$ be the corresponding coroots. Fix a nondegenerate invariant bilinear form $(,)$ in $\g$ such that $(\alpha_i^{\vee},\alpha_j^{\vee})=a_{i,j}/d_j$. Define the corresponding invariant bilinear forms in $\h^*$ such that $(\alpha_i,\alpha_j)=d_ia_{i,j}$. 
We have $\langle \lambda,\alpha_i^{\vee}\rangle=2(\lambda,\alpha_i)/(\alpha_i,\alpha_i)$ for $\lambda\in\h^*$. In particular, $\langle\alpha_j,\alpha_i^{\vee}\rangle=a_{i,j}$. Let $\omega_1,\dots,\omega_r\in\h^*$ be the fundamental weights, $\langle \omega_j,\alpha_i^{\vee}\rangle=\delta_{i,j}$.

Let $\mathcal P=\{\la\in\mathfrak h^*|\langle \la,\alpha_i^{\vee}\rangle\in\mathbb Z\}$ and $\mathcal P^+=\{\la\in\mathfrak h^*|\langle \la,\alpha_i^{\vee}\rangle\in\mathbb Z_{\gge 0}\}$ be the weight lattice and the set of dominant integral weights. The dominance order $>$ on $\mathfrak{h}^*$ is defined by: $\mu>\nu$ if and only if $\mu-\nu=\sum_{i=1}^r a_i\alpha_i$, $a_i\in \Z_{\gge 0}$ for $i=1,\dots,r$.

Let $\rho\in {\mathfrak h}^*$ be such that $\langle \rho,\alpha_i^{\vee}\rangle =1$, $i=1,\dots,r$.  We have $(\rho,\alpha_i)=(\alpha_i,\alpha_i)/2$.

For $\la\in {\mathfrak h}^*$, let $V_\la$ be the irreducible $\g$-module with highest weight $\la$. We
denote $\langle \la,\alpha_i^{\vee}\rangle$ by $\lambda_i$ and sometimes write $V_{(\la_1,\la_2,\dots,\la_r)}$ for $V_\la$.

The Weyl group $\mc W\subset\mathrm{Aut}(\mathfrak h^*)$ is generated by reflections $s_i$, $i=1,\dots,r$,
\[s_i(\la)=\la-\langle \la,\alpha_i^{\vee}\rangle \alpha_i,\quad \la\in\mathfrak h^*.\]We use the notation
\[w\cdot\la=w(\la+\rho)-\rho,\quad w\in\mc W,\la\in\mathfrak h^*,\]
for the shifted action of the Weyl group.

Let $E_1,\dots,E_r\in \n_+$, $H_1,\dots,H_r\in \h$, $F_1,\dots,F_r\in\n_-$ be the Chevalley generators of $\g$.

The coproduct $\Delta:\mc U(\g)\to \mc U(\g)\otimes \mc U(\g)$ is defined to be the homomorphism of algebras such that $\Delta x=1\otimes x+x\otimes 1$, for all $x\in \g$.

Let $(x_i)_{i\in O}$ be an orthonormal basis with respect to the bilinear form $(,)$ in $\g$.

Let $\Omega_0=\sum_{i\in O}x_i^2\in\mc U(\g)$ be the Casimir element. For any $u\in \mc U(\g)$, we have $u\Omega_0=\Omega_0 u$.
Let $\Omega=\sum_{i\in O}x_i\otimes x_i\in \g\otimes \g\subset \mc U(\g)\otimes\mc U(\g)$. For any $u\in \mc U(\g)$, we have $\Delta(u)\Omega=\Omega\Delta(u)$.

The following lemma is well-known, see for example \cite{H}, Ex. 23.4.
\begin{lem}\label{cas eigen}
Let $V_\la$ be an irreducible module of highest weight $\la$. Then $\Omega_0$ acts on $V_\la$ by the constant $(\la+\rho,\la+\rho)-(\rho,\rho)$.\qed
\end{lem}

Let $V$ be a $\g$-module.
Let $\mathrm{Sing}~V=\{v\in V~|~\n_+v=0\}$ be the subspace of singular vectors in $V$.
For $\mu\in\h^*$ let $V[\mu]=\{v\in V~|~hv=\langle \mu,h \rangle v\}$  be the subspace of $V$ of vectors of weight $\mu$.
Let $\mathrm{Sing}~V[\mu]=(\mathrm{Sing}~V)\cap (V[\mu])$ be the subspace of singular vectors in $V$ of weight $\mu$.

\subsection{Gaudin Model}
Let $n$ be a positive integer and $\bm{\La}=(\La_1,\dots,\La_n)$, $\La_i\in\h^*$, a sequence of weights. Denote by $V_{\bm \La}$ the $\g$-module $V_{\La_1}\otimes\dots\otimes V_{\La_n}$.

If $X\in\mathrm{End}(V_{\La_i})$, then we denote by $X^{(i)}\in\mathrm{End}(V_{\bm\La})$ the operator $\mathrm{id}^{\otimes i-1}\otimes X\otimes \mathrm{id}^{\otimes n-i}$ acting non-trivially on the $i$-th factor of the tensor product. If $X=\sum_k X_k\otimes Y_k\in \mathrm{End}(V_{\La_i}\otimes V_{\La_j})$, then we set $X^{(i,j)}=\sum_k X_k^{(i)}\otimes Y_k^{(j)}\in \mathrm{End}(V_{\bm\La})$.

Let $\bm z=(z_1,\dots,z_n)$ be a point in $\C^n$ with distinct coordinates. Introduce linear operators $\cH_1(\bm z),\dots,\cH_n(\bm z)$ on $V_{\bm\La}$ by the formula
\beq\label{eq hamil}
\cH_i(\bm z)=\sum_{j,\ j\ne i}\frac{\Omega^{(i,j)}}{z_i-z_j},\quad i=1,\dots,n.
\eeq
The operators $\cH_1(\bm z),\dots,\cH_n(\bm z)$ are called the \emph{Gaudin Hamiltonians} of the Gaudin model associated with $V_{\bm\La}$. One can check that the Hamiltonians commute, $[\cH_i(\bm z),\cH_j(\bm z)]=0$ for all $i,j$. Moreover, the Gaudin Hamiltonians commute with the action of $\g$, $[\cH_i(\bm z),x]=0$ for all $i$ and $x\in \g$. Hence for any $\mu\in\h^*$, the Gaudin Hamiltonians preserve the subspace $\mathrm{Sing}~ V_{\bm\La}[\mu]\subset V_{\bm\La}$.

\subsection{Bethe ansatz}\label{sec:Bethe ansatz}
Fix a sequence of weights $\bm{\Lambda}=(\Lambda_i)_{i=1}^n$, $\Lambda_i\in\h^*$, and a sequence of non-negative integers $\bm{l}=(l_1,\dots,l_r)$. Denote $l=l_1+\dots+l_r$, $\La=\La_1+\dots+\La_n$, and $\alpha(\bm l)=l_1\alpha_1+\dots+l_r\alpha_r$.

Let $c$ be the unique non-decreasing function from $\{1,\dots,l\}$ to $\{1,\dots,r\}$, such that $\# c^{-1}(i)=l_i$ for $i=1,\dots,r$. The \emph{master function} $\Phi(\bm t,\bm z,\bm\La,\bm l)$ is defined by
\[
\Phi(\bm t,\bm z,\bm\La,\bm l)=\prod_{1\lle i<j\lle n}(z_i-z_j)^{(\La_i,\La_j)}\prod_{i=1}^l\prod_{s=1}^n(t_i-z_s)^{-(\alpha_{c(i)},\La_s)}\prod_{1\lle i<j\lle l}(t_i-t_j)^{(\alpha_{c(i)},\alpha_{c(j)})}.
\]
The function $\Phi$ is a function of complex variables $\bm t=(t_1,\dots,t_l)$, $\bm z=(z_1,\dots,z_n)$, weights $\bm\La$, and discrete parameters $\bm l$. The main variables are $\bm t$, the other variables will be considered as parameters.

We call $\La_i$ {\it the weight at a point $z_i$}, and we also call $c(i)$ {\it the color of variable $t_i$}.

A  point $\bm t\in\C^l$ is called a \emph{critical point associated to $\bm z,\bm \La,\bm l$}, if the following system of algebraic equations is satisfied,
\beq\label{eq:bae}
-\sum_{s=1}^n\frac{(\alpha_{c(i)},\La_s)}{t_i-z_s}+\sum_{j,j\ne i}\frac{(\alpha_{c(i)},\alpha_{c(j)})}{t_i-t_j}=0,\quad i=1,\dots,l.
\eeq
In other words, a point $\bm t$ is a critical point if
\[
\left(\Phi^{-1}\frac{\pa\Phi}{\pa t_{i}}\right)(\bm t)=0,\quad \text{for}~ i=1,\dots,l.
\]
Equation (\ref{eq:bae}) is called the \emph{Bethe ansatz equation associated to $\bm \La, \bm z,\bm l$}. 

By definition, if $\bm t=(t_1,\dots,t_l)$ is a critical point and $(\alpha_{c(i)},\alpha_{c(j)})\ne 0$ for some $i,j$, then $t_i\ne t_j$. Also if $(\alpha_{c(i)},\La_s)\ne0$ for some $i,s$, then $t_i\ne z_s$.

Let $\Sigma_l$ be the permutation group of the set $\{1,\dots,l\}$. Denote by $\bm{\Sigma_l}\subset \Sigma_l$ the subgroup
of all permutations preserving the level sets of the function $c$. The subgroup $\bm{\Sigma_l}$ is
isomorphic to $\Sigma_{l_1}\times \dots \times \Sigma_{l_r}$. The action of the subgroup $\bm{\Sigma_l}$ preserves the set of critical points of the master function. All orbits of $\bm{\Sigma_l}$ on the critical set have the same cardinality $l_1 !\dots l_r!$. In what follows we do not distinguish between critical points in the same $\bm{\Sigma_l}$-orbit.

The following lemma is known.
\begin{lem}\cite{MV2} If weight $\La-\alpha(\bs l)$ is dominant integral, then the set of critical points is finite. $\qed$

\end{lem}

\subsection{Weight Function}Consider highest weight irreducible $\g$-modules $V_{\La_1},\dots,V_{\La_n}$, the tensor product $V_{\bm\La}$, and its weight subspace $V_{\bm\La}[\La-\alpha(\bm l)]$. Fix a highest weight vector $v_{\La_i}$ in $V_{\La_i}$ for $i=1,\dots,n$.

Following \cite{SV}, we consider a rational map
\[\omega:\C^n\times\C^l\to V_{\bm\La}[\La-\alpha(\bm l)]\]\
called \emph{the canonical weight function}.

Let $P(\bm l,n)$ be the set of sequences $I=(i_1^1,\dots,i_{j_1}^1;\dots;i_1^n,\dots,i_{j_n}^n)$ of integers in $\{1,\dots,r\}$ such that for all $i=1,\dots,r$, the integer $i$ appears in $I$ precisely $l_i$ times. For $I\in P(\bm l,n)$, and a permutation $\sigma\in\Sigma_l$, set $\sigma_1(i)=\sigma(i)$ for $i=1,\dots,j_1$ and $\sigma_s(i)=\sigma(j_1+\dots+j_{s-1}+i)$ for $s=2,\dots,n$ and $i=1,\dots,j_s$. Define
\[
\Sigma(I)=\{\sigma\in\Sigma_l~|~c(\sigma_s(j))=i_s^j~\text{for }s=1,\dots,n\text{ and }j=1,\dots,j_s\}.
\]
To every $I\in P(\bm l,n)$ we associate a vector
\[F_Iv=F_{i_1^1}\dots F_{i_{j_1}^1}v_{\La_1}\otimes\dots\otimes F_{i_1^n}\dots F_{i_{j_n}^n}v_{\La_n}\]
in $V_{\bm\La}[\La-\alpha(\bm l)]$, and rational functions
\[
\omega_{I,\sigma}=\omega_{\sigma_1(1),\dots,\sigma_1(j_1)}(z_1)\dots\omega_{\sigma_n(1),\dots,\sigma_n(j_n)}(z_n),
\]labeled by $\sigma\in\Sigma(I)$, where
\[\omega_{i_1,\dots,i_j}(z)=\frac{1}{(t_{i_1}-t_{i_2})\dots(t_{i_{j-1}}-t_{i_j})(t_{i_j}-z)}.\]

We set
\beq\label{wtf}
\omega(\bm z,\bm t)=\sum_{I\in P(\bm l,n)}\sum_{\sigma\in\Sigma(I)}\omega_{I,\sigma}F_Iv.
\eeq

Let $\bm t\in\C^l$ be a critical point of the master function $\Phi(\cdot,\bm z,\bm\La,\bm l)$. Then the value of the weight function $\omega(\bm z,\bm t)\in V_{\bs \La} [\La-\alpha(\bs l)]$ is called {\it the Bethe vector}. Note that the Bethe vector does not depend on a choice of the representative in the $\bm{\Sigma_l}$-orbit of critical points.

The following facts about Bethe vectors are known.
Assume that $\bm z\in \C^n$ has distinct coordinates.
Assume that $\bm t\in\C^l$ is an isolated critical point of the master function $\Phi(\cdot,\bm z,\bm\La,\bm l)$.

\begin{lem}\cite{MV1}
The Bethe vector $\omega(\bm z,\bm t)$ is well defined.\hfill$\square$
\end{lem}
\begin{thm}\cite{V}\label{thm:bvnonzero}
The Bethe vector $\omega(\bm z,\bm t)$ is non-zero.\hfill$\square$
\end{thm}
\begin{thm}\cite{RV}\label{thm:bveigen}
The Bethe vector $\omega(\bm z,\bm t)$ is singular,
$\omega(\bm z,\bm t)\in\mathrm{Sing}~V_{\bm\La}[\La-\alpha(\bm l)]$. Moreover, $\omega(\bm z,\bm t)$ is a common eigenvector of the Gaudin Hamiltonians,
$$\cH_i(\bm z)\omega(\bm z,\bm t)=\left(\Phi^{-1}\frac{\partial \Phi}{\partial z_i}\right)(\bm t,\bm z)\omega(\bm z,\bm t), \qquad i=1,\dots,n.$$ \hfill$\square$
\end{thm}

\subsection{Polynomials representing critical points}\label{sec2.5}
Let $\bm t=(t_1,\dots,t_l)$ be a critical point of a master function $\Phi(\bm t,\bm z,\bm\La,\bm l)$. Introduce a sequence of polynomials  $\bm y=(y_1(x),\dots,y_r(x))$ in a variable $x$ by the formula
\[
y_i(x)=\prod_{j,c(j)=i}(x-t_{j}).
\]
We say that the $r$-tuple of polynomials $\bm y$ \emph{represents a critical point $\bs t$} of the master function $\Phi(\bm t,\bm z,\bm\La,\bm l)$.
Note that the $r$-tuple $\bm y$ does not depend on a choice of the representative in the $\bm{\Sigma_l}$-orbit of the critical point $\bm t$.

We have $l=\sum_{i=1}^r\mathrm{deg}~y_i=\sum_{i=1}^rl_i$. We call $l$ {\it the length of $\bm y$}. We use notation $\bm y^{(l)}$ to indicate the length of $\bm y$.


Introduce functions
\beq\label{eq T}
T_i(x)=\prod_{s=1}^n(x-z_s)^{\langle\Lambda_s,\alpha_i^{\vee}\rangle},\quad i=1,\dots,r.
\eeq

We say that a given $r$-tuple of polynomials $\bm y\in\bm P(\C[x])^r$ is \emph{generic with respect to $\bs \La, \bs z$} if
\begin{description}
  \item[G1] polynomials $y_i(x)$ have no multiple roots;
  \item[G2] roots of $y_i(x)$ are different from roots and singularities of the function $T_i$;
  \item[G3] if $a_{ij}<0$ then polynomials $y_i(x)$, $y_j(x)$ have no common roots.
\end{description}

If $\bm y$ represents a critical point of $\Phi$, then $\bm y$ is generic.

\medskip

Following \cite{MV4}, we
reformulate the property of $\bm y$ to represent a critical point for the case when all but one weights are dominant integral.

We denote by $W(f,g)$ the Wronskian of functions $f$ and $g$, $W(f,g)=f'g-fg'$.

\begin{thm}\cite{MV4}\label{thm:2.5}
Assume that $\bs z\in\C^n$ has distinct coordinates and $z_1=0$. Assume that $\Lambda_i\in\mathcal P^+$, $i=2,\dots,n$.
A generic $r$-tuple $\bm y$ represents a critical point associated to $\bs \La,\bs z,\bs l$ if and only if for every $i=1,\dots,r$ there exists a polynomial $\tilde y_i$ satisfying
\beq\label{3}
W(y_i,x^{\langle \La_1+\rho,\alpha_i^\vee\rangle} \tilde y_i)=T_i\prod_{j\ne i}y_j^{-\langle \alpha_{j},\alpha_i^{\vee}\rangle}.
\eeq
Moreover, if the $r$-tuple $\bm {\tilde y_i}=(y_1,\dots,\tilde y_i,\dots,y_r)$ is generic, then it represents a critical point associated to data $(s_i\cdot\La_1,\La_2,\dots,\La_n), \bs z, \bm{l_i}$, where $\bm{l_i}$ is determined by equation $$\La-\La_1-\alpha(\bm{l_i})=s_i(\La-\La_1-\alpha(\bs l)).$$
\hfill$\square$
\end{thm}
We say that the $r$-tuple $\bm {\tilde y_i}$ (and the critical point it represents)
is obtained from the $r$-tuple $\bm  y$ (and the critical point it represents) by {\it the reproduction procedure in the $i$-th direction}.

Note that reproduction procedure can be iterated. The reproduction procedure in the $i$-th direction applied to $r$-tuple $\tl{\bm y}_i$ returns back the $r$-tuple $\bs y$. More generally, it is shown in \cite{MV4}, that the $r$-tuples obtained from $\bs y$ by iterating a reproduction procedure are in a bijective correspondence with the elements of the Weyl group.

We call a function $f(x)$ a \emph{quasi-polynomial} if it has the form $x^a p(x)$, where $a\in\C$ and $p(x)\in\C[x]$.
Under the assumptions of Theorem \ref{thm:2.5}, all $T_i$ are quasi-polynomials.

\section{Solutions of Bethe ansatz equation in the case of  \texorpdfstring{$V_{\la}\otimes V_{\omega_1}$}{Lg} for type \texorpdfstring{$\mathrm{B}_r$}{Lg}}\label{sec n=2, B}
In Sections \ref{sec n=2, B}, \ref{sec B generic} we work with Lie algebra of type $\mathrm{B}_r$.

Let $\g=\mathfrak{so}(2r+1)$. We have $(\alpha_i,\alpha_i)=4$, $i=1,\dots,r-1$, and $(\alpha_r,\alpha_r)=2$.

In this section we work with data $\bs \La=(\la,\omega_1)$, $\bs z=(0,1)$. The main result of the section
is the explicit formulas for the solutions of the Bethe ansatz equations, see Corollary \ref{thm sol B}.

\subsection{Parameterization of solutions}\label{sec2.6}
One of our goals is to diagonalize the Gaudin Hamiltonians associated to $\bs \La=(\la,\omega_1)$, $\bs z=(0,1)$.
It is sufficient to do that in the spaces of singular vectors of a given weight.

Let $\la\in\mathcal P^+$. We write the decomposition of finite-dimensional $\g$-module $V_\la\otimes V_{\omega_1}$. We have
\begin{align}\label{eq:dec of B}
V_{\la}\otimes V_{\omega_1}=&V_{\la+\omega_1}\oplus V_{\la+\omega_1-\alpha_1}\oplus \dots \oplus V_{\la+\omega_1-\alpha_1-\dots-\alpha_r}
\oplus V_{\la+\omega_1-\alpha_1-\dots-\alpha_{r-1}-2\alpha_r}\nonumber\\&\oplus V_{\la+\omega_1-\alpha_1-\dots-\alpha_{r-2}-2\alpha_{r-1}-2\alpha_r}\oplus\dots\oplus V_{\la+\omega_1-2\alpha_1-\dots-2\alpha_{r-1}-2\alpha_r}\nonumber\\
=&V_{(\la_1+1,\la_2,\dots,\la_r)}\oplus V_{(\la_1-1,\la_2+1,\la_3,\dots,\la_r)}\oplus
V_{(\la_1,\dots,\la_{k-1},\la_{k}-1,\la_{k+1}+1,\dots,\la_r)}\nonumber\\ &\oplus\dots\oplus V_{(\la_1,\la_2,\dots,\la_{r-2},\la_{r-1}-1,\la_r+2)}
\oplus V_{(\la_1,\la_2,\dots,\la_{r-1},\la_r)}\oplus V_{(\la_1,\la_2,\dots,\la_{r-2},\la_{r-1}+1,\la_r-2)} \nonumber\\ &\oplus V_{(\la_1,\dots,\la_{r-2}+1,\la_{r-1}-1,\la_r)}\oplus\dots\oplus V_{(\la_1-1,\la_2,\dots,\la_r)},
\end{align}
with the convention that the summands with non-dominant highest weights are omitted. In addition, if $\la_r=0$, then the summand $V_{\la-\alpha_1-\dots-\alpha_r}=V_{(\la_1,\la_2,\dots,\la_{r-1},\la_r)}$ is absent.

Note, in particular, that all multiplicities are 1.

By Theorem \ref{thm:bveigen}, to diagonalize the Gaudin Hamiltonians, it is sufficient to find a solution of the Bethe ansatz equation (\ref{eq:bae}) associated to $\bs\La,\bs z$ and $\bs l$ corresponding to the summands in the decomposition (\ref{eq:dec of B}).
We call an $r$-tuple of integers $\bs l$ {\it admissible} if $V_{\la+\omega_1-\alpha(\bs l)}\subset V_\la\otimes V_{\omega_1}$.

The admissible $r$-tuples $\bs l$ have the form
\begin{align}\label{admit}
\bs l=(\underbrace{1,\dots,1}_{k \text{ ones}},0,\dots,0) \qquad {\rm or} \qquad \bs l=(\underbrace{1,\dots,1}_{k \text{ ones}},2,\dots,2),
\end{align}
where $k=0,\dots,r$. In the first case the length $l=l_1+\dots +l_r$ is $k$ and in the second case $2r-k$. It follows that
different admissible $r$-tuples have different length and, therefore, admissible tuples $\bs l$ are parameterized by length $l\in \{0,1,\dots,2r\}$. We call a nonnegative integer $l$ {\it admissible} if it is the length of an admissible  $r$-tuple $\bs l$.
More precisely, a nonnegative integer $l$ is  admissible if $l=0$ or if $l\lle r$, $\la_l>0$ or if $l=r+1$, $\la_r>1$ or if $r+1<l\lle 2r$, $\la_{2r-l+1}>0$.

In terms of $\bm y=(y_1,\dots,y_r)$, we have the following cases, corresponding to \eqref{admit}.

For $l\lle r$, the polynomials $y_1,\dots,y_{l}$ are linear and $y_{l+1},\dots,y_r$ are all equal  to one.

For $l>r$, the polynomials $y_1,\dots,y_{2r-l}$ are linear and $y_{2r-l+1},\dots,y_r$ are quadratic.

\begin{rem}\label{rem type A}
For $l\lle r$ the Bethe ansatz equations for type $\mathrm{B}_r$ coincide with the Bethe ansatz equations for type $\mathrm{A}_r$ which were solved directly in \cite{MV3}.
In what follows, we recover the result for $l<r$, and we refer to \cite{MV3} for the case of $l=r$.
\end{rem}

\subsection{Example of \texorpdfstring{$\mathrm{B}_2$}{Lg}}\label{sec3.3} We illustrate our approach in the case of $\mathrm{B}_2$, $l=4$.
We have $n=2$, $\La_1=\la\in\mathcal P^+$, $\La_2=\omega_1$, $z_1=0,z_2=1$. We write $\la=(\la_1,\la_2)$, where $\la_i=\langle\la,\alpha_i^\vee\rangle\in\Z_{\gge 0}$.

Suppose the Bethe ansatz equation has a solution with $l=4$. Then it is represented by quadratic polynomials $y_1^{(4)}$ and $y_2^{(4)}$. By Theorem \ref{thm:2.5}, it means that there exist polynomials $\tilde y_1,\tilde y_2$ such that
\begin{align*}
W(y_1^{(4)},\tilde y_1)=x^{\la_1}(x-1)y_2^{(4)},\qquad W(y_2^{(4)},\tilde y_2)=x^{\la_2}\big(y_1^{(4)}\big)^2.
\end{align*}

Note we have $\la_1,\la_2\in\Z_{\gge 0}$, but  for $\la_1=0$ the first equation is impossible for degree reasons.
Therefore, there are no solutions with $l=4$ for $\la_1=0$ which is exactly when the corresponding summand is absent
in \eqref{eq:dec of B} and when $l=4$ is not admissible.

\textbf{Step 1}: There exists a unique monic linear polynomial $u_1$ such that $-\la_1\tilde y_1=x^{\la_1+1}u_1$.
Clearly, the only root of $u_1$ cannot coincide with the roots of $x^{\la_1}(x-1)y_2^{(4)}$, therefore the pair $(u_1,y_2^{(4)})$ is generic.
It follows from Theorem \ref{thm:2.5}, that the pair $(u_1,y_2^{(4)})$ solves Bethe ansatz equation with $l=3$ and $\la$ replaced by $s_1\cdot\la=(-\la_1-2,2\la_1+\la_2+2)$.

In terms of Wronskians, it means that there exist quasi-polynomials $\hat y_1$ and $\hat y_2$ such that
\begin{align*}
W(u_1,\hat y_1)=x^{-\la_1-2}(x-1)y_2^{(4)},\qquad W(y_2^{(4)},\hat y_2)=x^{2\la_1+\la_2+2}u_1^2.
\end{align*}

The procedure we just described corresponds to the reproduction in the first direction, we have
$s_1(\omega_1-2\alpha_1-2\alpha_2)=\omega_1-\alpha_1-2\alpha_2$.

Note that $s_2(\omega_1-2\alpha_1-2\alpha_2)=\omega_1-2\alpha_1-2\alpha_2$ and the reproduction in the second direction applied to $(y_1^{(4)},y_2^{(4)})$  does not change $l=4$. We do not use it.

\textbf{Step 2}: We apply the reproduction in the second direction to the $l=3$ solution $(u_1,y_2^{(4)})$.

By degree reasons, we have $-(\la_2+2\la_1+1)\hat y_2=x^{\la_2+2\la_1+3}\cdot 1$. Set $u_2=1$. Clearly, the pair $(u_1,u_2)$ is generic.
By Theorem \ref{thm:2.5}, the pair $(u_1,u_2)$ solves Bethe ansatz equation with $l=1$ and  $\La_1=(s_2s_1)\cdot\la=(\la_1+\la_2+1,-2\la_1-\la_2-4)$.

It means, we have
$s_2(\omega_1-\alpha_1-2\alpha_2)=\omega_1-\alpha_1$ and there exist quasi-polynomials $\bar y_1,\bar y_2$ such that \
\[ W(u_1,\bar y_1)=x^{\la_1+\la_2+1}(x-1)u_2=x^{\la_1+\la_2+1}(x-1),\qquad W(u_2,\bar y_2)= x^{-2\la_1-\la_2-4}u_1^2.\]

Note that we also have $\la_1\hat y_1=x^{-\la_1-1}y_1^{(4)}$. Therefore, we can recover the initial solution $(y_1^{(4)}, y_2^{(4)})$ from $(u_1,y_2^{(4)})$. In general, if we start with an arbitrary $l=3$ solution and use the reproduction in the first direction, we obtain a pair of quadratic polynomials. If this pair is generic, then it represents an $l=4$ solution associated to the data $\bm\La=(\la,\omega_1),\bm z,\bm l=(2,2)$. However, we have no easy argument to show that it is generic. Thus, our procedure gives an inclusion of all $l=4$ solutions to the $l=3$ solutions and we need an extra argument to show this inclusion is a bijection.

\textbf{Step 3}: Finally, we apply the reproduction in the first direction to the $l=1$ solution $(u_1,u_2)$.

We have $-(\la_1+\la_2+1)\bar y_1=x^{\la_1+\la_2+2}\cdot 1$. Set $v_1=1$.
Clearly, the pair $(v_1,u_2)=(1,1)$ is generic and represents the only solution of the Bethe ansatz equation
with $l=0$ and   $\La_1=(s_1s_2s_1)\cdot\la=(-\la_1-\la_2-3,\la_2)$.  We denote the final weight $(s_1s_2s_1)\cdot\la$ by $\theta=(\theta_1,\theta_2)$.

It means, we have
$s_1(\omega_1-\alpha_1)=\omega_1$, and there exist quasi-polynomials $\overset{\circ}{y}_1,\overset{\circ}{y}_2$ such that
\[W(v_1,\overset{\circ}{y}_1)= x^{-\la_1-\la_2-3}(x-1)u_2^2,\qquad W(u_2,\overset{\circ}{y_2})=x^{\la_2}v_1=x^{\la_2}.\]

As before, we have $(\la_1+\la_2+1)\overset{\circ}{y}_1=x^{-\la_1-\la_2-2}u_1$, and therefore using reproduction in the first direction to pair $(v_1, u_2)$ we recover the pair $(u_1,u_2)$.

\medskip

To sum up, we have the inclusions of solutions for $l=4$ to $l=3$ to $l=1$ to $l=0$ with the $\La_1$ varying by the shifted action of the Weyl group. Since for $l=0$ the solution is unique, it follows that for $l=1,3,4$ the solutions are at most unique. Moreover, if it exists, it can be computed recursively.

\medskip

We proceed with the direct computation of $y_1^{(4)},y_2^{(4)}$. From step 3, we have $v_1=u_2=1$. Then we compute
\[u_1=x-\frac{\la_1+\la_2+1}{\la_1+\la_2+2}.\]

From step 2, we get
\[
y_2^{(4)}=x^2-\frac{2(2\la_1+\la_2+1)(\la_1+\la_2+1)}{(2\la_1+\la_2+2)(\la_1+\la_2+2)}x+\frac{(2\la_1+\la_2+1)(\la_1+\la_2+1)^2}{(2\la_1+\la_2+3)(\la_1+\la_2+2)^2}.
\]Finally, from Step 1,
\begin{align*}
y_1^{(4)}=&x^2-\frac{(2\la_1+\la_2+1)(2\la_1^2+2\la_1\la_2+4\la_1+\la_2+2)}{(\la_1+1)(\la_1+\la_2+2)(2\la_1+\la_2+2)}x\\&+\frac{\la_1(\la_1+\la_2+1)(2\la_1+\la_2+1)}{(\la_1+1)(\la_1+\la_2+2)(2\la_1+\la_2+3)}.
\end{align*}
From the formula it is easy to check that the pair $(y_1^{(4)},y_2^{(4)})$ is generic if $\la_1>0$ and therefore represents a solution of the Bethe ansatz equation associated to $\bm \La,\bm z$ and $l=4$.

Thus the Bethe ansatz equation associated to $\bm \La,\bm z,\bm l=(2,2)$ has a unique solution given by the formulas above.

\subsection{The recursion lemmas}\label{sec lemmas}

Let $l\in\{0,\dots,r-1,r+2,\dots,2r\}$, we establish a reproduction procedure which produces solutions of length $l-1$ from the ones of length $l$. For $l=r+1$, the reproduction procedure goes from $l=r+1$ to $l=r-1$. We recover the special case $l=r$ directly from \cite{MV3}, see Remark \ref{rem type A}. By Theorem \ref{thm:2.5} it is sufficient to check that the new $r$-tuple of polynomial is generic with respect to new data. It is done with the help of following series of lemmas.

For brevity, we denote $x-1$ by $y_0$ for this section.

The first lemma describes the reproduction in the $k$-th direction from $l=2r-k+1$ to $l=2r-k$, where $k=1,\dots,r-1$.
\begin{lem}\label{lem3.4}
Let $k\in\{1,\dots, r-1\}$. Let $\nu=(\nu_1,\dots,\nu_r)$ be an integral weight such that $\nu_{k}\gge 0$. Let $y_1,\dots,y_{k-1}$ be linear polynomials and $y_{k},\dots,y_{r}$ be quadratic polynomials. Suppose the $r$-tuple of polynomials $\bm y^{(2r-k+1)}=(y_1,\dots,y_r)$ represents a critical point associated to $(\nu,\omega_1),\bm z$ and $l=2r-k+1$. Then there exists a unique monic linear polynomial $u_{k}$ such that $W(y_k,x^{\nu_k+1}u_k)=-\nu_k x^{\nu_k}y_{k-1}y_{k+1}$. Moreover, $\nu_{k}> 0$ and the $r$-tuple of polynomials $\bm y^{(2r-k)}=(y_1,\dots,y_{k-1},u_k,y_{k+1},\dots,y_r)$ represents a critical point associated to $(s_k\cdot\nu,\omega_1),\bm z$ and $l=2r-k$.
\end{lem}
\begin{proof}
The existence of polynomial $\tl y_k$ such that $W(y_k,\tl y_k)=x^{\nu_k}y_{k-1}y_{k+1}$ implies $\nu_k>0$. Indeed, if $\deg\tl y_k\gge 3$, then $\deg W(y_k,\tl y_k)\gge 4$; if $\deg\tl y_k\lle 2$, then $\deg W(y_k,\tl y_k)\lle 2$. Hence $\deg x^{\nu_k}y_{k-1}y_{k+1}\ne 3$, it follows that $\nu_k\ne 0$.

By Theorem \ref{thm:2.5}, it is enough to show $\bm y^{(2r-k)}$ is generic. If $y_{k-1}y_{k+1}$ is divisible by $u_k$, then $y_k$ has common root with $y_{k-1}y_{k+1}$ which is impossible since $(y_1,\dots,y_r)$ is generic. Since $u_k$ is linear, it cannot have a multiple root.
\end{proof}
Note that we do not have such a lemma for the reproduction in the $k$-th direction which goes from $l-1$ to $l$ since unlike $u_k$ the new polynomial is quadratic and we cannot immediately conclude that it has distinct roots. We overcome this problem using the explicit formulas in Section \ref{sec check generic}.

The next lemma describes the reproduction in the $r$-th direction from $l=r+1$ to $l=r-1$.
\begin{lem}\label{lem3.12}
Let $\nu=(\nu_1,\dots,\nu_r)$ be an integral weight such that $\nu_{r}\gge 0$. Let $y_1,\dots,y_{r-1}$ be linear polynomials and $y_r$ be a quadratic polynomial. Suppose the $r$-tuple of polynomials $\bm y^{(r+1)}=(y_1,\dots,y_r)$ represents a critical point associated to $(\nu,\omega_1),\bm z$ and $l=r+1$. Then $W(y_{r},x^{\nu_{r}+1})=-(\nu_{r}-1)x^{\nu_{r}}y_{r-1}^2$. Moreover, $\nu_{r}\gge 2$ and the $r$-tuple of polynomials $\bm y^{(r-1)}=(y_1,\dots,y_{r-2},y_{r-1},1)$ represents a critical point associated to $(s_{r}\cdot\nu,\omega_1),\bm z$ and $l=r-1$.\hfill$\square$
\end{lem}

Finally, we disuss the reproduction in the $k$-th direction from $l=k$ to $l=k-1$, where $k=1,\dots,r-1$.
\begin{lem}\label{lem3.7}
Let $k\in \{1,\dots, r-1\}$. Let $\nu=(\nu_1,\dots,\nu_r)$ be an integral weight such that $\nu_{k}\gge 0$. Let $y_1,\dots,y_k$ be linear polynomials and $y_{k+1}=\dots=y_r=1$. Suppose the $r$-tuple of polynomials $\bm y^{(k)}=(y_1,\dots,y_r)$ represents a critical point associated to $(\nu,\omega_1),\bm z$ and $l=k$. Then  $W(y_{k},x^{\nu_{k}+1})=-\nu_{k}x^{\nu_{k}}y_{k-1}y_{k+1}$. Moreover, $\nu_{k}> 0$ and the $r$-tuple of polynomials $\bm y^{(k-1)}=(y_1,\dots,y_{k-1},1,1,\dots,1)$ represents a critical point associated to $(s_{k}\cdot\nu,\omega_1),\bm z$ and $l=k-1$.\hfill$\square$
\end{lem}

\subsection{At most one solution}\label{sec at most one} In this section, we show that there exists at most one solution of the Bethe ansatz equation (\ref{eq:bae}).

We start with the explicit formulas for the shifted action of the Weyl group.
\begin{lem}\label{lem3.10} Let $\la=(\la_1,\dots,\la_r)\in\h^*$.

We have
\[(s_1\dots s_k)\cdot \la=(-\la_1-\dots-\la_k-k-1,\la_1,\dots,\la_{k-1},\la_k+\la_{k+1}+1,\la_{k+2},\dots,\la_r),\]
where $k=1,\dots, r-2$,
\[(s_1\dots s_{r-1})\cdot \la=(-\la_1-\dots-\la_{r-1}-r,\la_1,\dots,\la_{r-2},2\la_{r-1}+\la_r+2),\]
\[(s_1\dots s_{r})\cdot \la=(-\la_1-\dots-\la_{r}-r-1,\la_1,\dots,\la_{r-2},2\la_{r-1}+\la_r+2),\]
and
\begin{align*}
(s_1\dots s_{r}s_{r-1}\dots s_{2r-k})\cdot \la=(&-\la_1-\dots-\la_{2r-k-1}-2\la_{2r-k}-\dots-2\la_{r-1}-\la_r-k-1,\\&\la_1,\dots,\la_{2r-k-2},\la_{2r-k-1}+\la_{2r-k}+1,\la_{2r-k+1},\dots,\la_r),
\end{align*}
where $k=r+1,\dots, 2r-1$.
\end{lem}

\begin{proof}
If $k=1,\dots, r-2, r$, the action of a simple reflection is given by
\[s_k\cdot\la=(\la_1,\dots,\la_{k-2},\la_{k-1}+\la_k+1,-\la_k-2,\la_k+\la_{k+1}+1,\la_{k+2},\dots,\la_r).\]
In addition,
\[
s_{r-1}\cdot\la=(\la_1,\dots,\la_{r-3},\la_{r-2}+\la_{r-1}+1,-\la_{r-1}-2,2\la_{r-1}+\la_r+2).
\]The lemma follows.\end{proof}

We also prepare the inverse formulas.
\begin{lem}\label{lem3.9} Let $\theta=(\theta_1,\dots,\theta_r)\in\h^*$.
We have
\[
(s_k\dots s_1)\cdot\theta=(\theta_2,\dots,\theta_k,-\theta_1-\dots-\theta_k-k-1,\theta_1+\dots+\theta_{k+1}+k,\theta_{k+2},\dots,\theta_r),
\] where $k=1,\dots, r-2$,
\[
(s_{r-1}\dots s_1)\cdot\theta=(\theta_2,\dots,\theta_{r-1},-\theta_1-\dots-\theta_{r-1}-r,2\theta_1+\dots+2\theta_{r-1}+\theta_r+2r-2 ).
\]
\[
(s_r\dots s_1)\cdot\theta=(\theta_2,\dots,\theta_{r-1},\theta_1+\dots+\theta_{r}+r-1,-2\theta_1-\dots-2\theta_{r-1}-\theta_r-2r),
\] and
\begin{align*}
(s_{2r-k}s_{2r-k+1}&\dots s_rs_{r-1}\dots s_1)\cdot \theta\\=(&\theta_2,\dots,\theta_{2r-k-1},\theta_1+\dots+\theta_{2r-k-1}+2\theta_{2r-k}+\dots+2\theta_{r-1}+\theta_r+k-1,\\&-\theta_1-\dots-\theta_{2r-k}-2\theta_{2r-k+1}-\dots-2\theta_{r-1}-\theta_r-k,\theta_{2r-k+1},\dots,\theta_r),
\end{align*}
where $k=r+1,\dots, 2r-1$.
In particular,
$$(s_1s_2\dots s_rs_{r-1}\dots s_1)\cdot\theta=(-\theta_1-2\theta_2-\dots-2\theta_{r-1}-\theta_r-2r+1,\theta_2,\dots,\theta_r).$$
\qed \end{lem}

\begin{lem}\label{lem3.8}
Let $\la\in\cP^+$ and let $\bs l$ be as in \eqref{admit}. Suppose the Bethe ansatz equation associated to $\bm \La=(\la,\omega_1),\bm z=(0,1),\bm l$, where $\la\in\cP^+$, has solutions. Then $\bm l$ is admissible. Moreover, if $l\gge r+1$, then we can perform the reproduction procedure in the $(2r-l+1)$-th, $(2r-l+2)$-th, $\dots$, $(r-1)$-th, $r$-th, $(r-1)$-th, $\dots$, $1$-st directions successively. If $l<r$, we can perform the reproduction procedure in the $l$-th, $(l-1)$-th, $\dots$, $1$-st directions successively.
\end{lem}
\begin{proof}
We use Lemmas \ref{lem3.4}-\ref{lem3.7}. The condition of the lemmas of the form $\nu_k\gge 0$ follows from Lemmas \ref{lem3.10} and \ref{lem3.9}.
\end{proof}
\begin{cor}
Let $\la\in\cP^+$ and $\bs l$ as in \eqref{admit}. The Bethe ansatz equation \eqref{eq:bae} associated to $\bm \La,\bm z, \bs l$, has at most one solution. If $l$ is not admissible it has no solutions.
\end{cor}
\begin{proof} If $l\neq r$, then by Lemma \ref{lem3.8}, every solution of the Bethe ansatz equations by a series of reproduction procedures produces a solution for $l=0$. These reproduction procedures are invertible, and for $l=0$ we clearly have only one solution $(1,\dots,1)$. Therefore the conclusion.

For $l=r$ the corollary follows from Theorem 2 in \cite{MV3}, see also Remark \ref{rem type A}.
\end{proof}

\subsection{Explicit solutions}
In this section, we give explicit formulas for the solution of the Bethe ansatz equation corresponding to data $\bs \La=(\la,\omega_1)$, $\bs z=(0,1)$ and $l$, $\la\in\mathcal P^+$, $l\in\{0,\dots,2r\}$.

We denote by $\theta$ the weight obtained from $\la$ after the successive reproduction procedures as in Lemma \ref{lem3.8}.
Explicitly, if $l\lle r-1$, then $\theta=(s_1\dots s_{l-1}s_l)\cdot\la$; if $l\gge r+1$, then $\theta=(s_1\dots s_{r-1}s_rs_{r-1}\dots s_{2r-l+1})\cdot\la$. We recover the solution starting from data $(\theta,\omega_1)$, $z=(0,1)$ and $l=0$, where the solution is $(1,\dots,1)$ by applying the reproduction procedures in the opposite direction explicitly. In the process we obtain monic polynomials $(y_1^{(l)},\dots,y_r^{(l)})$ representing a critical point.

Recall that for $l\lle r$, $y_1,\dots,y_{l}$ are linear polynomials and $y_{l+1},\dots,y_r$ are all equal  to one.
We use the notation: $y_i^{(l)}=x-c_i^{(l)}$, $i=1,\dots,l$.

Recall further that for $l>r$, the polynomials  $y_1,\dots,y_{2r-l}$ are linear and $y_{2r-l+1},\dots,y_r$ are quadratic.
We use the notation: $y_i^{(l)}=x-c_i^{(l)}$, $i=1,\dots,2r-l$ and $y_i^{(l)}=(x-a_i^{(l)})(x-b_i^{(l)})$, $i=2r-l+1,\dots,r$.

Formulas for $c_i^{(l)}$, $a_i^{(l)}$ and $b_i^{(l)}$ in terms of $\theta_i$, clearly, do not depend on $l$, in such cases we simply write $c_i$, $a_i$ and $b_i$.

Denote $y_0^{(k)}=x-1$, $c_0=1$ and $T_1(x)=x^{\la_1}$. Also let
$$
A^{(k)}(\theta)=\begin{cases} (s_k\dots s_1)\cdot\theta &\mbox{if } k\lle r,\\
(s_{2r-k}\dots s_{r-1}s_rs_{r-1}\dots s_1)\cdot\theta &\mbox{if } k\gge r+1. \end{cases}
$$
Explicitly, $A^{(k)}(\theta)$ are given in Lemma \ref{lem3.9}.

\subsubsection{Constant term of $y_i$ in terms of $\theta$} For brevity, we write simply $A^{(k)}$ for $A^{(k)}(\theta)$. We also use $A^{(k)}_i$ for components of the weight $A^{(k)}$: $A^{(k)}=(A^{(k)}_1,\dots,A^{(k)}_r)$.

For $l\lle r-1$, we have $\bm y^{(l-1)}=(x-c_1,\dots,x-c_{l-1},1,\dots,1)$. It is easy to check that if $l$ is admissible and $\la$ is dominant then $A^{(l-1)}_l=\theta_1+\dots+\theta_l+(l-1)$ is a negative integer.

We solve for $\tl{y}_l^{(l-1)}$,
\[
W(y_l^{(l-1)},\tl{y}_l^{(l-1)})=T_l^{(l-1)}y_{l-1}^{(l-1)}y_{l+1}^{(l-1)}=x^{A_l^{(l-1)}}(x-c_{l-1}).
\]In other words
\[-(\tl{y}_l^{(l-1)})'=x^{A_l^{(l-1)}+1}-c_{l-1}x^{A_l^{(l-1)}}.
\]Choosing the solution which is a quasi-polynomial, we obtain
\begin{align*}\tl{y}_l^{(l-1)}=&\frac{-x^{A_l^{(l-1)}+1}}{A_l^{(l-1)}+2}\left(x-\frac{A_l^{(l-1)}+2}{A_l^{(l-1)}+1}c_{l-1}\right).\end{align*}
Therefore, the reproduction procedure in the $l$-th direction gives $\bm y^{(l)}=(x-c_1,\dots,x-c_{l},1,\dots,1)$, where $c_l=\dfrac{A_l^{(l-1)}+2}{A_l^{(l-1)}+1}c_{l-1}$.
Substituting the value for $A_l^{(l-1)}$ and using induction, we have
$$
c_k=\prod_{j=1}^k\frac{\theta_1+\dots+\theta_j+j+1}{\theta_1+\dots+\theta_j+j},
$$
for $k=1,\dots, r-1$.

For $l=r+1$ we have $\bm y^{(r-1)}=(x-c_1,\dots,x-c_{r-1},1)$ and $A_r^{(r-1)}=2\theta_1+\dots+2\theta_{r-1}+\theta_r+2r-2\in\Z_{<0}$.
We solve for $\tl{y}_r^{(r-1)}$,
\[W(y_r^{(r-1)},\tl{y}_r^{(r-1)})=T_r^{(r-1)}(y_{r-1}^{(r-1)})^2=x^{A_r^{(r-1)}}(x-c_{r-1})^2.\]
This implies
\[\tl{y}_r^{(r-1)}=\frac{-x^{A_r^{(r-1)}+1}}{A_r^{(r-1)}+3}\left(x^2-\frac{2(A_r^{(r-1)}+3)}{A_r^{(r-1)}+2}c_{r-1}x+\frac{A_r^{(r-1)}+3}{A_r^{(r-1)}+1}c_{r-1}^2\right) .\]
Therefore, after performing the reproduction procedure in $r$-th direction to $\bs y^{(r-1)}$, we obtain the $r$-tuple
$\bm y^{(r+1)}=(x-c_1,\dots,x-c_{r-1},(x-a_r)(x-b_r))$, where
\begin{align}\label{3.4.2}
a_rb_r=\frac{A_r^{(r-1)}+3}{A_r^{(r-1)}+1}c_{r-1}^2
=\left(\prod_{j=1}^{r-1}\frac{\theta_1+\dots+\theta_j+j+1}{\theta_1+\dots+\theta_j+j}\right)^2\frac{2\theta_1+\dots+2\theta_{r-1}+\theta_r+2r+1}{2\theta_1+\dots+2\theta_{r-1}+\theta_r+2r-1}.\nonumber
\end{align}

For $l$ such that $r+2\leqslant l\leqslant 2r$, let $k=2r-l$, then $\bm y^{(2r-k)}=(x-c_1,\dots,x-c_{k},(x-a_{k+1})(x-b_{k+1}),\dots,(x-a_r)(x-b_r))$ and $A_{k}^{(2r-k-1)}=\theta_1+\dots+\theta_{k}+2\theta_{k+1}+\dots+2\theta_{r-1}+\theta_r+2r-k-2\in\Z_{<0}$.

We have \[W(y_k^{(2r-k)},\tl{y}_k^{(2r-k)})=x^{A_k^{(2r-k-1)}}y_{k-1}^{(2r-k)}y_{k+1}^{(2r-k)},\] substituting $-(A_k^{(2r-k-1)}+2)\tl{y}_k^{(2r-k)}=x^{A_k^{(2r-k-1)}+1}(x-a_k)(x-b_k)$, we get
\begin{align}\label{3.4.3}
&(A_{k}^{(2r-k-1)}+1)(x-c_k)(x-a_k)(x-b_k)+x(x-c_k)(x-a_k)\nonumber\\&+x(x-c_k)(x-b_k)-x(x-a_k)(x-b_k)\\
=&(A_k^{(2r-k-1)}+2)(x-a_{k+1})(x-b_{k+1})(x-c_{k-1}).\nonumber
\end{align}
Substituting $x=0$ into (\ref{3.4.3}), we obtain
\beq\label{3.4.4}
(A_{k}^{(2r-k-1)}+1)c_ka_k b_k=(A_k^{(2r-k-1)}+2)c_{k-1}a_{k+1}b_{k+1}.
\eeq
It results in
\begin{align}\label{3.4.7}
a_kb_k=&c_{r-1}c_{k-1}\frac{2\theta_1+\dots+2\theta_{r-1}+\theta_r+2r+1}{2\theta_1+\dots+2\theta_{r-1}+\theta_r+2r-1}\nonumber\\
&\times\prod_{i=1}^{r-k}\frac{(\theta_1+\dots+\theta_{r-1})+(\theta_{r}+\theta_{r-1}+\dots+\theta_{r+1-i})+r+i}{(\theta_1+\dots+\theta_{r-1})+(\theta_{r}+\theta_{r-1}+\dots+\theta_{r+1-i})+r+i-1}\nonumber\\
=&\frac{2\theta_1+\dots+2\theta_{r-1}+\theta_r+2r+1}{2\theta_1+\dots+2\theta_{r-1}+\theta_r+2r-1}\prod_{j=1}^{r-1}\frac{\theta_1+\dots+\theta_j+j+1}{\theta_1+\dots+\theta_j+j}
\prod_{j=1}^{k-1}\frac{\theta_1+\dots+\theta_j+j+1}{\theta_1+\dots+\theta_j+j}
\nonumber\\&\times\prod_{i=1}^{r-k}\frac{(\theta_1+\dots+\theta_{r-1})+(\theta_{r}+\theta_{r-1}+\dots+\theta_{r+1-i})+r+i}{(\theta_1+\dots+\theta_{r-1})+(\theta_{r}+\theta_{r-1}+\dots+\theta_{r+1-i})+r+i-1}.
\end{align}
\subsubsection{The formula for $a_k+b_k$ in terms of $\theta$}
Comparing the coefficient of $x^2$ in (\ref{3.4.3}), we obtain
\beq\label{3.4.8}
(A_{k}^{(2r-k-1)}+1)(a_k+b_k+c_k)+2c_k=(A_{k}^{(2r-k-1)}+2)(a_{k+1}+b_{k+1}+c_{k-1}).
\eeq
Comparing the coefficient of $x$ in (\ref{3.4.3}), we obtain
\begin{align}\label{3.4.9}
&(A_{k}^{(2r-k-1)}+1)(c_k(a_k+b_k)+a_kb_k)+c_k(a_k+b_k)-a_kb_k\nonumber
\\=&(A_{k}^{(2r-k-1)}+2)(c_{k-1}(a_{k+1}+b_{k+1})+a_{k+1}b_{k+1}).
\end{align}

Solving (\ref{3.4.8}) and (\ref{3.4.9}) for $a_k+b_k$, one has
$$
a_k+b_k=\frac{(A_{k}^{(2r-k-1)}+2)(a_{k+1}b_{k+1}-c_{k-1}^2)+(A_{k}^{(2r-k-1)}+3)c_{k-1}c_k-A_{k}^{(2r-k-1)}a_kb_k}{(A_{k}^{(2r-k-1)}+2)c_k-(A_{k}^{(2r-k-1)}+1)c_{k-1}}.
$$
This gives the explicit formulas,
\begin{align*}
a_k+b_k=&\frac{2\theta_1+\dots+2\theta_{r-1}+2r+1}{2\theta_1+\dots+2\theta_{r-1}+2r}\prod_{j=1}^{k-1}\frac{\theta_1+\dots+\theta_j+j+1}{\theta_1+\dots+\theta_j+j}\nonumber\\
&\times\left(1+\prod_{j=k}^{r-1}\frac{\theta_1+\dots+\theta_j+j+1}{\theta_1+\dots+\theta_j+j}\times\right.\nonumber\\
&\left.\prod_{j=k}^{r-1}\frac{\theta_1+\dots+\theta_j+2\theta_{j+1}+\dots+2\theta_{r-1}+\theta_r+2r-k}{\theta_1+\dots+\theta_j+2\theta_{j+1}+\dots+2\theta_{r-1}+\theta_r+2r-k-1}\right).
\end{align*}
These solutions indeed satisfy (\ref{3.4.3}) for each $k$. This can be checked by a direct computation.
\subsubsection{Final formulas}\label{sec3.5.3} We use Lemma \ref{lem3.10} to express $\theta_i$ by $\la_j$. Here are the final formulas.

If $l<r$, then
\beq\label{3.4.12}
c_j^{(l)}=\prod_{i=1}^j\frac{\la_i+\dots+\la_l+l-i}{\la_i+\dots+\la_l+l-i+1},\quad j=1,\dots,l.
\eeq

We also borrow from \cite{MV3} the $l=r$ result.

\beq\label{3.4.16}
c_j^{(r)}=\prod_{i=1}^j\frac{\la_i+\dots+\la_{r-1}+\la_r/2+r-i}{\la_i+\dots+\la_{r-1}+\la_r/2+r-i+1},\quad j=1,\dots,r.
\eeq

If $l\gge r+1$, then
\beq\label{3.4.13}
c_k^{(l)}=\prod_{j=1}^k\frac{\la_j+\dots+\la_{2r-l}+2\la_{2r-l+1}+\dots+2\la_{r-1}+\la_r+l-j-1}{\la_j+\dots+\la_{2r-l}+2\la_{2r-l+1}+\dots+2\la_{r-1}+\la_r+l-j},
\eeq
for $k=1,\dots,2r-l$. Finally, for $2r-l+1\lle k\lle r$, we have
\begin{align}\label{3.4.14}
a_k^{(l)}b_{k}^{(l)}=&\left(\prod_{j=1}^{2r-l}\frac{\la_j+\dots+\la_{2r-l}+2\la_{2r-l+1}+\dots+2\la_{r-1}+\la_r+l-j-1}{\la_j+\dots+\la_{2r-l}+2\la_{2r-l+1}+\dots+2\la_{r-1}+\la_r+l-j}\right)^2\nonumber\\
&\times \prod_{j=2r-l+1}^{r-1}\frac{\la_{2r-l+1}+\dots+\la_j+2\la_{j+1}+\dots+2\la_{r-1}+\la_r+l-j-2}{\la_{2r-l+1}+\dots+\la_j+2\la_{j+1}+\dots+2\la_{r-1}+\la_r+l-j-1}\nonumber\\
&\times\prod_{j=2r-l+1}^{k-1}\frac{\la_{2r-l+1}+\dots+\la_j+2\la_{j+1}+\dots+2\la_{r-1}+\la_r+l-j-2}{\la_{2r-l+1}+\dots+\la_j+2\la_{j+1}+\dots+2\la_{r-1}+\la_r+l-j-1}\nonumber\\
&\times\prod_{i=1}^{r-k}\frac{\la_{2r-l+1}+\dots+\la_{r-i}+l-r-i-1}{\la_{2r-l+1}+\dots+\la_{r-i}+l-r-i}\nonumber\\
&\times\frac{2\la_{2r-l+1}+\dots+2\la_{r-1}+\la_r+2l-2r-3}{2\la_{2r-l+1}+\dots+2\la_{r-1}+\la_r+2l-2r-1},
\end{align}
and
\begin{align}\label{3.4.15}
a_k^{(l)}+b_{k}^{(l)}=&\frac{2\la_{2r-l+1}+\dots+2\la_{r-1}+\la_r+2l-2r-3}{2\la_{2r-l+1}+\dots+2\la_{r-1}+\la_r+2l-2r-2}\nonumber\\
&\times \prod_{j=1}^{2r-l}\frac{\la_j+\dots+\la_{2r-l}+2\la_{2r-l+1}+\dots+2\la_{r-1}+\la_r+l-j-1}{\la_j+\dots+\la_{2r-l}+2\la_{2r-l+1}+\dots+2\la_{r-1}+\la_r+l-j}\nonumber\\
&\times\left(\prod_{j=2r-l+1}^{k-1}\frac{\la_{2r-l+1}+\dots+\la_j+2\la_{j+1}+\dots+2\la_{r-1}+\la_r+l-j-2}{\la_{2r-l+1}+\dots+\la_j+2\la_{j+1}+\dots+2\la_{r-1}+\la_r+l-j-1}\right.\nonumber\\
&+\prod_{j=2r-l+1}^{r-1}\frac{\la_{2r-l+1}+\dots+\la_j+2\la_{j+1}+\dots+2\la_{r-1}+\la_r+l-j-2}{\la_{2r-l+1}+\dots+\la_j+2\la_{j+1}+\dots+2\la_{r-1}+\la_r+l-j-1}\nonumber\\
&\times \left.\prod_{i=1}^{r-k}\frac{\la_{2r-l+1}+\dots+\la_{r-i}+l-r-i-1}{\la_{2r-l+1}+\dots+\la_{r-i}+l-r-i}\right).
\end{align}

\subsubsection{The solutions are generic}\label{sec check generic}
In this section we show the solutions are generic.

\begin{thm}\label{thm B2 generic}
Suppose $\la\in\cP^+$ and $\bm l$ is admissible, then $\bm y^{(l)}$ in Section \ref{sec3.5.3} represents a critical point associated to $\bs\La=(\la,\omega_1)$, $\bs z=(0,1)$, and $\bm l$.
\end{thm}
\begin{proof}
It is sufficient to show that $\bs y$ is generic with respect to $\bm \La,\bm z$.

Let us first consider {\bf G2}. For $\la\in\mathcal{P}^+$, {\bf G2} is equivalent to $y_1^{(l)}(1)\ne 0$ and $y_i^{(l)}(0)\ne 0$ if $\la_i\ne 0$.

If $l\lle r-1$, then the admissibility of $\bm l$ implies $\la_l>0$. To prove {\bf G2}, it suffices to show $c_j^{(l)}\ne 0$ if $\la_l\ne 0$ and $c_1^{(l)}\ne 1$, see  (\ref{3.4.12}). Note that if $\la_l>0$, then
\[
0<\frac{\la_i+\dots+\la_l+l-i}{\la_i+\dots+\la_l+l-i+1}<1
\]for all $i\in\{1,\dots,l\}$, therefore all $c_j^{(l)}\in (0,1)$.

If $l=r$, this is similar to the previous situation.

If $l=r+1$, the admissibility of $\bm l$ implies $\la_r\gge 2$. {\bf G2} is obviously true.

If $l\gge r+2$,  the admissibility of $\bm l$ implies $\la_{2r-l+1}>0$. One has $y_k^{(l)}(0)\ne 0$ since we have $a_k^{(l)}b_k^{(l)}\ne 0$. As for $y_1^{(l)}(1)\ne 0$ in the case $l=2r$, we delay the proof until after the case {\bf G1}.

Now, we consider {\bf G1}. Suppose $a_k^{(l)}=b_k^{(l)}$ for some $2r-l+1\lle k\lle r$. Observe that
\[
W(y_k^{(l)},\tl y_k^{(l)})=T_k^{(l)}y_{k-1}^{(l)}y_{k+1}^{(l)}.
\]
By {\bf G2}, $y_k^{(l)}$ and $T_k^{(l)}$ have no common roots. In addition if $l=2r$ and $k=1$, we have $y_{1}^{(l)}(1)=0$, then $a_{1}^{(l)}b_{1}^{(l)}=1$, while as above we have $a_{1}^{(l)}b_{1}^{(l)}\in(0,1)$. It follows that we must have
$a_k^{(l)}=a_{k+1}^{(l)}$ or $a_k^{(l)}=a_{k-1}^{(l)}$($a_k^{(l)}=c_{k-1}^{(l)}$, if $k=2r-l+1$).

We work in terms of $\theta$. We have $a_k=b_k=a_{k+1}$ or $a_k=b_k=a_{k-1}$ or $a_{2r-l+1}=c_{2r-l}$. If $a_k=b_k=a_{k+1}$, then substituting $x=c_k$ into (\ref{3.4.3}), we get
\beq\label{3.4.19}
-c_k(c_k-a_{k+1})=(c_k-b_{k+1})(c_k-c_{k-1})(A_{k}^{(2r-k-1)}+2).
\eeq
Solving (\ref{3.4.4}) and (\ref{3.4.19}) for $a_{k+1}=a_k=b_{k}$ and $b_{k+1}$ in terms of $c_k$, $c_{k-1}$ and $A_{k}^{(2r-k-1)}$, we obtain
\[
a_{k+1}b_{k+1}=(A_{k}^{(2r-k-1)}+2)(A_{k}^{(2r-k-1)}+1)c_{k-1}c_{k}\left(\frac{(A_{k}^{(2r-k-1)}+3)c_k-(A_{k}^{(2r-k-1)}+2)c_{k-1}}{(A_{k}^{(2r-k-1)}+1)c_k-A_{k}^{(2r-k-1)}c_{k-1}}\right)^2.
\]Comparing it with (\ref{3.4.7}) and canceling common factors, we obtain
\begin{align*}
&(A_{k}^{(2r-k-1)}+2)(A_{k}^{(2r-k-1)}+1)\frac{2\theta_1+\dots+2\theta_{r-1}+\theta_r+2r+1}{2\theta_1+\dots+2\theta_{r-1}+\theta_r+2r-1}\\
=&\prod_{i=k}^r\frac{\theta_1+\dots+\theta_i+i+1}{\theta_1+\dots+\theta_i+i}\prod_{i=k+2}^{r-1}\frac{\theta_1+\dots+\theta_{i-1}+2\theta_i+\dots+2\theta_{r-1}+\theta_r+2r-i+1}{\theta_1+\dots+\theta_{i-1}+2\theta_i+\dots+2\theta_{r-1}+\theta_r+2r-i}.
\end{align*}
Substituting $\theta_i$ in terms of $\la_j$, we have
\begin{align}\label{3.10.1}
&(\la_{2r-l+1}+\dots+\la_k+k+l-2r-1)(\la_{2r-l+1}+\dots+\la_k+k+l-2r)\nonumber\\&\times\frac{2\la_{2r-l+1}+\dots+2\la_{r-1}+\la_r+2l-2r-3}{2\la_{2r-l+1}+\dots+2\la_{r-1}+\la_r+2l-2r-1}\nonumber\\
=&\prod_{j=k}^{r-1}\frac{\la_{2r-l+1}+\dots+\la_j+2\la_{j+1}+\dots+2\la_{r-1}+\la_r+l-j-2}{\la_{2r-l+1}+\dots+\la_j+2\la_{j+1}+\dots+2\la_{r-1}+\la_r+l-j-1}\nonumber\\&\times \prod_{i=1}^{r-k-1}\frac{\la_{2r-l+1}+\dots+\la_{r-i}+l-r-i-1}{\la_{2r-l+1}+\dots+\la_{r-i}+l-r-i}.
\end{align}
By our assumption, we have $\la_{2r-l+1}\gge 1$, $k\gge 2r-l+1$ and $l\gge r+2$. It is easily seen that
\begin{align*}
&(\la_{2r-l+1}+\dots+\la_k+k+l-2r-1)(\la_{2r-l+1}+\dots+\la_k+k+l-2r)\\&\times\frac{2\la_{2r-l+1}+\dots+2\la_{r-1}+\la_r+2l-2r-3}{2\la_{2r-l+1}+\dots+2\la_{r-1}+\la_r+2l-2r-1}\gge 1\times 2\times \frac{3}{5}> 1.
\end{align*}
Therefore (\ref{3.10.1}) is impossible. Similarly, we can exclude $a_k^{(l)}=a_{k-1}^{(l)}$. As for $a_{2r-l+1}^{(l)}=c_{2r-l}^{(l)}$, by (\ref{3.4.4}), it is impossible since each fractional factor is strictly less than $1$.

Finally, we prove {\bf G3}. The nontrivial cases are $a_k^{(l)}=a_{k+1}^{(l)}$ and $a_{2r-l+1}^{(l)}=c_{2r-l}^{(l)}$ for $l\gge r+1$, where $k\gge 2r-l+1$.

If $a_k=a_{k+1}$, then by (\ref{3.4.3}) we have that $x-a_k$ divides $x(x-c_k)(x-b_k)$. As we already proved $a_k\ne b_k$ and $a_k\ne 0$, it follows that $a_k=c_k$. This again implies that $(x-a_k)^2$ divides $(x-a_{k+1})(x-b_{k+1})(x-c_{k-1})$ as $A_k^{(2r-k+1)}+2\ne 0$. If $b_{k+1}=a_k=a_{k+1}$, then we are done. If $a_k=c_{k-1}$, then $c_{k-1}=c_k$. It is impossible by the argument used in {\bf G2}.

If $a_k=c_{k-1}$, then by (\ref{3.4.3}) one has $x-a_k$ divides $x(x-c_k)(x-b_k)$.
Since $l$ is admissible, $a_k^{(l)}\ne0$. Then $a_k\ne b_k$ implies $c_{k-1}=c_k$. It is also a contradiction.

In particular, this shows that $y_1^{(l)}$ and $y_0^{(l)}$ have no common roots, i.e., $y_1^{(l)}(1)\ne 0$.
\end{proof}
\begin{cor}\label{thm sol B}
Suppose $\la\in\cP^+$. Then the Bethe ansatz equation \eqref{eq:bae} associated to $\bm \La,\bm z, \bm l$, where $\bm l$ is admissible, has exactly one solution. Explicitly, for $l\lle r-1$, the corresponding $r$-tuple $\bm y^{(l)}$ which represents the solution is described by \eqref{3.4.12}, for  $l=r$ by \eqref{3.4.16}, for $2r\gge l\gge r+1$ by \eqref{3.4.13}, \eqref{3.4.14} and \eqref{3.4.15}.\hfill$\square$
\end{cor}

\subsection{Associated differential operators for type B}\label{sec oper}
Let $\bs y$ be an $r$-tuple of quasi-polynomials. Following \cite{MV2}, we introduce a linear differential operator $D(\bm y)$ of order $2r$ by the formula
\begin{align*}
D_\la(\bm y)=&\left(\pa-\ln'\left(\frac{T_1^2\dots T_{r-1}^2 T_r}{y_1}\right)\right)\left(\pa-\ln'\left(\frac{y_1T_1^2\dots T_{r-1}^2T_r}{y_2T_1}\right)\right)\\
&\times \left(\pa-\ln'\left(\frac{y_2T_1^2\dots T_{r-1}^2 T_r}{y_3T_1T_2}\right)\right)\dots \left(\pa-\ln'\left(\frac{y_{r-1}T_1\dots T_{r-1} T_r}{y_r}\right)\right)\\
&\times \left(\pa-\ln'\left(\frac{y_{r}T_1\dots T_{r-1} }{y_{r-1}}\right)\right)\left(\pa-\ln'\left(\frac{y_{r-1}T_1\dots T_{r-2} }{y_{r-2}}\right)\right)\dots\\
&\times (\pa-\ln'(y_1)),
\end{align*}
where $T_i$, $i=1,\dots,r$, are given by \eqref{eq T}.

If $\bs y$ is an $r$-tuple of polynomials representing a critical point associated to integral dominant weights $\La_1,\dots,\La_n$ and points $z_1,\dots,z_n$ of type $\mathrm{B}_r$, then by \cite{MV2}, the kernel of $D_\la(\bm y)$ is a self-dual space of polynomials. By \cite{MM} the coefficients of $D_\la(\bm y)$ are eigenvalues of higher Gaudin Hamiltonians acting on the Bethe vector related to $\bs y$.

For admissible $l$ and $\la\in \mathfrak{h}^*$, define $a_{\la}^l(1),\dots, a_\la^l(r)$ as the following.

For $l=0,\dots,r-1$, $i=1,\dots, l$, set $a_{\la}^l(i)=\la_i+\dots+\la_l+l+1-i$.
For $l=0,\dots,r-1$, $i=l+1,\dots, r$, set $a_{\la}^l(i)=0$.

For $l=r+1,\dots,2r$, set $k=2r-l$. Then for $i=1,\dots, k$, set $$a_{\la}^l(i)=\la_i+\dots+\la_k+2\la_{k+1}+\dots+2\la_{r-1}+\la_r+2r-k-i$$ and for  $i=k+1,\dots, r$, set $a_{\la}^l(i)=2\la_{k+1}+\dots+2\la_{r-1}+\la_r+2r-2k-1$.

\begin{prop}\label{diff B} Let the r-tuple $\bm y$ represent the solution of the Bethe ansatz equation \eqref{eq:bae} associated to $\bm\La$, $\bm z$ and admissible $\bm l$, where $\la\in\cP^+$ and $l\ne r$. Then $D_\la(\bm y)=D_\la(x^{a_{\la}^l(1)},\dots,x^{a_{\la}^l(r)})$.
\end{prop}
\begin{proof} The $(2r-1)$-tuple $(y_1,\dots,y_{r-1},y_r,y_{r-1},\dots, y_1)$ represents a critical point of type $\mathrm{A}_{2r-1}$. Then the reproduction procedure in direction $i$ of type $\mathrm{B}_r$ corresponds to a composition of reproduction procedures of type $\mathrm{A}_{2r-1}$ in directions $i$ and $2r-i$ for $i=1,\dots,r-1$, and to reproduction procedure of type $\mathrm{A}_{2r-1}$ in direction $r$ for $i=r$, see \cite{MV2}, \cite{MV4}.
Proposition follows from Lemma 4.2 in \cite{MV4}.
\end{proof}

\section{Completeness of Bethe ansatz for type B}\label{sec B generic}
In this section we continue to study the case of $\g={\mathfrak{so}}(2r+1)$.
The main result of the section is Theorem \ref{thm B generic}.

\subsection{Completeness of Bethe ansatz for \texorpdfstring{$V_\lambda\otimes V_{\omega_1}$}{Lg} }\label{sec3.6} Let $\la\in\mathcal P^+$.
Consider the tensor product of a finite-dimensional irreducible module with highest weight $\la$, $V_\lambda$, and the vector representation $V_{\omega_1}$.

Recall that the value of the weight function $\omega(z_1,z_2,\bm t)$ at a solution of the Bethe ansatz equations (\ref{eq:bae}) is called the Bethe vector.
We have the following result, which is usually referred to as completeness of Bethe ansatz.
\begin{thm}\label{thm3.14}
The set of Bethe vectors $\omega(z_1,z_2,\bm t)$, where $\bm t$ runs over the solutions to the Bethe ansatz equations \eqref{eq:bae} with admissible length $l$, forms a basis of ${\mathrm{Sing}}~(V_\lambda\otimes V_{\omega_1})$.
\end{thm}

\begin{proof}
All multiplicities in the decomposition of $V_\lambda\otimes V_{\omega_1}$ are 1. By Corollary \ref{thm sol B} for each admissible length $l$ we have a solution of the Bethe ansatz equation. The theorem follows from Theorems \ref{thm:bvnonzero} and \ref{thm:bveigen}.
\end{proof}

\subsection{Simple Spectrum of Gaudin Hamiltonians for \texorpdfstring{$V_\lambda\otimes V_{\omega_1}$}{Lg} }

We have the following standard fact.
\begin{lem}\label{hum}
Let $\mu,\nu\in\mathcal P^+$. If $\mu>\nu$ then $(\mu+\rho,\mu+\rho)>(\nu+\rho,\nu+\rho)$.
\end{lem}
\begin{proof}
The lemma follows from the proof of Lemma 13.2B in \cite{H}.
\end{proof}

\begin{prop}\label{prop3.18}
Let $\omega,\omega'\in V_\lambda\otimes V_{\omega_1}$ be Bethe vectors corresponding to solutions
to the Bethe ansatz equations of two different lengths. Then $\omega,\omega'$ are eigenvectors of the Gaudin Hamiltonian $\cH:=\cH_1=-\cH_2$
with distinct eigenvalues.
\end{prop}
\begin{proof}
Recall the relation
\[
\Omega^{(1,2)}=\frac{1}{2}\left(\Delta\Omega_0-1\otimes\Omega_0-\Omega_0\otimes 1\right).
\]Since $\Omega_0$ acts as a constant in
any irreducible module, $1\otimes\Omega_0+\Omega_0\otimes 1$ acts as a constant on $V_\lambda\otimes V_{\omega_1}$. It remains to consider the spectrum of the diagonal action  of $\Delta\Omega_0$. By Theorem \ref{thm:bveigen}, $\omega$ and $\omega'$ are highest weight vectors of
two non-isomorphic irreducible submodules of $V_\lambda\otimes V_{\omega_1}$. By Lemmas \ref{cas eigen} and \ref{hum} the values of $\Delta\Omega_0$ on $\omega$ and $\omega'$ are different.
\end{proof}

\subsection{The generic case}
We use the following well-known lemma from algebraic geometry.
\begin{lem}\label{lem5.1}Let $n\in \mathbb Z_{\gge 1}$ and suppose $f_k^{(\epsilon)}(x_1,\dots,x_l)=0$, $k=1,\dots,n$, is a system of $n$ algebraic
equations for $l$ complex variables $x_1,\dots,x_l$, depending on a complex parameter $\epsilon$ algebraically.
Let $(x_1^{(0)},\dots,x_l^{(0)})$ be an isolated solution with $\epsilon=0$. Then for sufficiently small $\epsilon$, there exists an
isolated solution $(x_1^{(\epsilon)},\dots,x_l^{(\epsilon)})$, depending algebraically on $\epsilon$, such that
\[x_k^{(\epsilon)}=x_k^{(0)}+o(1).\]\hfill$\square$
\end{lem}
Our main result is the following theorem.
\begin{thm}\label{thm B generic}
Let $\g=\mathfrak{so}(2r+1)$, $\la\in\cP^+$ and $N\in \Z_{\gge 0}$. For a generic $(N+1)$-tuple of distinct complex numbers $\bm z=(z_0,z_1,\dots,z_N)$, the Gaudin Hamiltonians $(\cH_0,\cH_1,\dots,\cH_N)$ acting in ${\rm Sing}\left(V_\la\otimes V_{\omega_1}^{\otimes N}\right)$ are diagonalizable and have simple joint spectrum. Moreover, for generic $\bm z$ there exists a set of solutions $\{\bm t_i,~i\in I\}$ of the Bethe ansatz equation \eqref{eq:bae} such that the corresponding Bethe vectors $\{\omega(\bm z,\bm t_i),~i\in I\}$ form a basis of ${\rm Sing}\left(V_\la\otimes V_{\omega_1}^{\otimes N}\right)$.
\end{thm}
\begin{proof}
Our proof follows that of Theorem 5.2 of \cite{MVY}, see also of Section 4 in \cite{MV1}.

Pick distinct non-zero complex numbers $\tl z_1,\dots,\tl z_N$. We use Theorem \ref{thm3.14} to define a basis in the space of singular vectors ${\rm Sing} (V_\la\otimes V_{\omega_1}^{\otimes N})$ as follows.

We call a $(k+1)$-tuple of weights $\mu_0,\mu_1,\dots,\mu_k\in\mathcal P^+$ {\it admissible} if $\mu_0=\la$ and for $i=1,\dots,k$, we have a submodule $V_{\mu_i}\subset V_{\mu_{i-1}}\otimes V_{\omega_1}$, see \eqref{eq:dec of B}.

For an admissible tuple of weights, we define a singular vector $v_{\mu_0,\dots,\mu_k}\in V_\la\otimes V_{\omega_1}^{\otimes k}$ of weight  $\mu_k$ using induction on $k$ as follows. Let $v_{\mu_0}=v_\la$ be the highest weight vector for module $V_\la$. Let $k$ be such that $1\lle k\lle N$. Suppose we have the singular vector $v_{\mu_0,\dots,\mu_{k-1}}\in V_\la\otimes V_{\omega_1}^{\otimes k-1}$. It generates a submodule $V_{\mu_0,\dots,\mu_{k-1}}\subset V_\la\otimes V_{\omega_1}^{\otimes k-1}$ of highest weight $\mu_{k-1}$.

Let $\bar{\bs t}_k=(\bar{t}_{k,j}^{(b)})$, where $b=1,\dots, r$ and $j=1,\dots,l_{k,b}$, be the solution of the Bethe ansatz equation
associated to $V_{\mu_{k-1}}\otimes V_{\omega_1}$, $\bs z=(0,\tilde z_k)$ and $\bs l_k=(l_{k,1},\dots,l_{k,r})$ such that $\mu_{k-1}+\omega_1-\alpha(\bs l_k)=\mu_k$.
Note that $\bar{\bs t}_k$ depends on $\mu_{k-1}$ and $\mu_k$, even though we do not indicate this dependence explicitly.
Note also that in all cases $l_{k,b}\in\{0,1,2\}$.

Then, define $v_{\mu_0,\dots,\mu_{k}}$ to be the
Bethe vector $$v_{\mu_0,\dots,\mu_{k}}=\omega(0,\tl z_k,\bar {\bs t}_k)\in V_{\mu_0,\dots,\mu_{k-1}}\otimes V_{\omega_1}\subset V_\la\otimes V_{\omega_1}^{\otimes k}.$$
We denote by $V_{\mu_0,\dots,\mu_k}$ the submodule of $V_\la\otimes V_{\omega_1}^{\otimes k}$ generated by $v_{\mu_0,\dots,\mu_k}$.

The vectors $v_{\mu_0,\dots,\mu_N}\in V_\la\otimes V_{\omega_1}^{\otimes N}$ are called the \emph{iterated singular vectors}.
To each iterated singular vector $v_{\mu_0,\dots,\mu_N}$ we have an associated collection
$\bar{\bm t}=(\bar{\bs t}_1,\dots,\bar{\bs t}_N)$
consisting of all the Bethe roots used in its construction.

Clearly, the iterated singular vectors corresponding to all admissible $(N+1)$-tuples of weights form a basis in ${\rm Sing} (V_\la\otimes V_{\omega_1}^{\otimes N})$, so we have
$$
V_\la\otimes V_{\omega_1}^{\otimes N}=\bigoplus_{\mu_0,\dots,\mu_N}V_{\mu_0, \mu_1,\dots,\mu_N},
$$
where the sum is over all admissible $(N+1)$-tuples of weights.

To prove the theorem, we show that  in some region
of parameters $\bm z$ for any admissible $(N+1)$-tuple of weights $\mu_0,\dots,\mu_N$,
there exists a Bethe vector $\omega_{\mu_1,\dots,\mu_N}$ which tends to $v_{\mu_1,\dots,\mu_N}$
when approaching a certain point (independent on $\mu_i$) on the boundary of the region.

To construct the Bethe vector $\omega_{\mu_1,\dots,\mu_N}$ associated to $v_{\mu_1,\dots,\mu_N}$,
we need to find a solution to the Bethe equations associated to  $V_\la\otimes V_{\omega_1}^{\otimes N}$ with Bethe roots,
 ${\bs t}=({t}_{j}^{(b)})$, where $b=1,\dots, r$ and $j=1,\dots, \sum_{k=1}^N l_{k,b}$.

We do it for $\bs z$ of the form
\beq\label{5.4}
z_0=z,\quad\text{and}\quad z_k=z+\ve^{N+1-k}\tl z_k,\quad k=1,\dots, N,
\eeq
for sufficiently small $\ve\in\C^\times$. Here $z\in\C$ is an arbitrary fixed number and $\tilde z_k$ are as above.

Then, similarly to $\bar{\bs t}$ we write
 $\bs t=(\bs t_1,\dots,\bs t_N)$ where $\bs t_k=(t_{k,j}^{(b)})$, $b=1,\dots,r$ and $j=1,\dots, l_{k,b}$, is constructed in the form
\beq\label{5.5}
t_{k,j}^{(b)}=z+\ve^{N+1-k}\tl t_{k,j}^{(b)},\quad k=1,\dots,N,~j=1,\dots, l_{k,b},~b=1,\dots, r.
\eeq
The variables $t_{k,j}^{(b)}$ satisfy the system of Bethe ansatz equations:
\beq\label{5.3a}
-\frac{(\la,\alpha_b)}{t_{k,j}^{(b)}-z_0}+\sum_{s=1}^N\left(\frac{-2\delta_{b,1}}{t_{k,j}^{(b)}-z_s}+\sum_{\substack{q=1\\ (s,q)\ne (k,j)}}^{l_{s,b}} \frac{(\alpha_b,\alpha_b)}{t_{k,j}^{(b)}-t_{s,q}^{(b)}}+\sum_{q=1}^{l_{s,b+1}}\frac{(\alpha_b,\alpha_{b+1})}{t_{k,j}^{(b)}-t_{s,q}^{(b+1)}}+\sum_{q=1}^{l_{s,b-1}}\frac{(\alpha_b,\alpha_{b-1})}{t_{k,j}^{(b)}-t_{s,q}^{(b-1)}}\right)=0
\eeq
for $b=1,\dots,r$, $k=1,\dots, N$, $j=1,\dots,l_{k,b}$. Here we agree that $l_{s,0}=l_{s,{N+1}}=0$ for all $s$.

Consider the leading asymptotic behavior of the Bethe ansatz equations as $\ve\to 0$.
We claim that in the leading order, the Bethe ansatz equations for $\bs t$ reduce to the Bethe ansatz equations obeyed by the
variables $\bm{\bar t}$.

Consider for example the leading order of the Bethe equation for $t_{k,j}^{(1)}$. Note that
\[
\frac{(\la,\alpha_1)}{t_{k,j}^{(1)}-z_0}+\sum_{s=1}^N\frac{2}{t_{k,j}^{(1)}-z_s}=\left(\frac{(\la,\alpha_1)}{\tl t_{k,j}^{(1)}}+\frac{2(k-1)}{\tl t_{k,j}^{(1)}}+\frac{2}{\tl t_{k,j}^{(1)}-\tl z_k}+\mc O(\ve)\right)\ve^{-N-1+k},
\]
\[
\sum_{s=1}^N\sum_{\substack{q=1\\ (s,q)\ne (k,j)}}^{l_{s,1}}\frac{(\alpha_1,\alpha_1)}{t_{k,j}^{(1)}-t_{s,q}^{(1)}}=\left(\sum_{\substack{q=1 \\ q\ne j}}^{l_{k,1}}\frac{(\alpha_1,\alpha_1)}{\tl t_{k,j}^{(1)}-\tl t_{k,q}^{(1)}}+\sum_{s=1}^{k-1}\sum_{q=1}^{l_{s,1}}\frac{(\alpha_1,\alpha_1)}{\tl t_{k,j}^{(1)}}+\mc O(\ve)\right)\ve^{-N-1+k},
\]and similarly
\[
\sum_{s=1}^N\sum_{q=1}^{l_{s,2}}\frac{(\alpha_1,\alpha_2)}{t_{k,j}^{(1)}-t_{s,q}^{(2)}}=\left(\sum_{q=1}^{l_{k,2}}\frac{(\alpha_1,\alpha_2)}{\tl t_{k,j}^{(1)}-\tl t_{k,q}^{(2)}}+\sum_{s=1}^{k-1}\sum_{q=1}^{l_{s,2}}\frac{(\alpha_1,\alpha_2)}{\tl t_{k,j}^{(1)}}+\mc O(\ve)\right)\ve^{-N-1+k}.
\]Then by definition of the numbers $l_{s,b}$, we have
\[
\mu_{k-1}=\la+(k-1)\omega_1-\sum_{b=1}^r\sum_{s=1}^{k-1}\sum_{q=1}^{l_{s,b}}\alpha_b
\]and, in particular,
\[(\mu_{k-1},\alpha_1)=(\la,\alpha_1)+2(k-1)-\sum_{s=1}^{k-1}\left(\sum_{q=1}^{l_{s,1}}(\alpha_1,\alpha_1)-\sum_{q=1}^{l_{s,2}}(\alpha_1,\alpha_2)\right).
\]Therefore
$$
-\frac{(\mu_{k-1},\alpha_1)}{\tl t_{k,j}^{(1)}}-\frac{2}{\tl t_{k,j}^{(1)}-\tl z_k}+\sum_{\substack{ q=1\\ q\ne j}}^{l_{k,1}}\frac{(\alpha_1,\alpha_1)}{\tl t_{k,j}^{(1)}-\tl t_{k,q}^{(1)}}+\sum_{q=1}^{l_{k,2}}\frac{(\alpha_1,\alpha_2)}{\tl t_{k,j}^{(1)}-\tl t_{k,q}^{(2)}}=\mc O(\ve).
$$
At leading order this is indeed the Bethe equation for $\bar t_{k,j}^{(1)}$ from the set of Bethe equations for the tensor product $V_{\mu_{k-1}}\otimes V_{\omega_1}$, with the tensor factors assigned to the points $0$ and $\tl z_k$, respectively. The other equations work similarly.

By Lemma \ref{lem5.1} it follows that for sufficiently small $\ve$ there exists a solution to the Bethe equations (\ref{5.3a}) of the form $\tl t_{j,k}^{(a)}=\bar t_{j,k}^{(a)}+o(1)$.

Now we claim that the Bethe vector $\omega_{\mu_1,\dots,\mu_N}=\omega(\bs z,\bs t)$ associated to $\bs t$ has leading asymptotic behavior
\beq\label{5.7}
\omega_{\mu_1,\dots,\mu_N}=\ve^K(v_{\mu_1,\dots,\mu_N}+o(1)),
\eeq
as $\ve \to 0$, for some $K$. Consider the definition (\ref{wtf}) of $\omega(\bs z,\bs t)$. We write $\omega_{\mu_1,\dots,\mu_N}=w_1+w_2$
where $w_1$ contains only those summands in which every factor in the denominator is of the form
\[
t_{k,j}^{(a)}-t_{k,q}^{(b)}\quad \text{or}\quad t_{k,j}^{(a)}-z_k.
\]
The term $w_2$ contains terms where at least one factor is
 of the form $t_{k,j}^{(a)}-t_{s,q}^{(b)}$ or $t_{k,j}^{(a)}-z_s$, $s\ne k$.
After substitution using (\ref{5.4}) and (\ref{5.5}), one finds that
\[
w_1=\left(\prod_{k=1}^N\prod_{j=1}^{r}\left(\ve^{-N-1+k}\right)^{l_{k,j}}\right)v_{\mu_1,\dots,\mu_N},
\]
and that $w_2$ is subleading to $w_1$, which establishes our claim.

Consider two distinct Bethe vectors $\omega_{\mu_1,\dots,\mu_N}$ and $\omega_{\mu'_1,\dots,\mu'_N}$ constructed as above. By Theorem
\ref{thm:bveigen} both are simultaneous eigenvectors of the quadratic Gaudin Hamiltonians $\cH_0,\cH_1,\dots,\cH_N$. Let $k$, be the largest possible number in $\{1,\dots,N\}$ such that $\mu_i=\mu'_i$ for all $i=1,\dots, k-1$. Consider the Hamiltonian $\cH_k$. When the $z_i$
are chosen as in (\ref{5.4}) then one finds
\beq\label{5.8}
\cH_k=\ve^{-N-1+k}\left(\sum_{j=0}^{k-1}\frac{\Omega^{(k,j)}}{\tl z_k}+o(1)\right).
\eeq
The sum $\sum_{j=0}^{k-1}\frac{\Omega^{(k,j)}}{\tl z_k}$ coincides with the action  of the quadratic Gaudin Hamiltonian $\cH$ of the spin chain $V_{\mu_{k-1}}\otimes V_{\omega_1}$ with sites at $0$ and $\tl z_k$,  embedded in $V_\la\otimes (V_{\omega_1})^{\otimes k}$ via
$$V_{\mu_{k-1}}\otimes V_{\omega_1}\simeq V_{\mu_1,\dots,\mu_{k-1}}\otimes V_{\omega_1}\subset V_\la\otimes (V_{\omega_1})^{\otimes k}.$$
Since $\mu_k\ne\mu'_k$, $v_{\mu_1,\dots,\mu_k}$ and $v_{\mu'_1,\dots,\mu'_k}$ are eigenvectors of $\sum_{j=0}^{k-1}\frac{\Omega^{(k,j)}}{\tl z_k}$ with distinct eigenvalues by Proposition \ref{prop3.18}. By (\ref{5.7}) and (\ref{5.8}), we have that the eigenvalues of $\cH_k$ on $\omega_{\mu_1,\dots,\mu_N}$ and $\omega_{\mu'_1,\dots,\mu'_N}$ are distinct.

The argument above establishes that the set of points $\bm z = (z_0,z_1,\dots, z_N )$ for which the Gaudin
Hamiltonians are diagonalizable with joint simple spectrum is non-empty. It is a Zariski-open set, therefore the theorem follows.
\end{proof}

\section{The cases of ${\rm C}_r$ and ${\rm D}_r$}\label{sec C and D}
\subsection{The case of ${\rm C}_r$}\label{C sec}
Let $\g=\mathfrak{sp}(2r)$, be the simple Lie algebra of type ${\rm C}_r$, $r\gge 3$. We have $(\alpha_i,\alpha_i)=2$, $i=1,\dots,r-1$, and $(\alpha_r,\alpha_r)=4$. We work with data $\bm \La=(\la,\omega_1)$, $\bm z=(0,1)$, where $\la\in\mathcal P^+$.

We have
\begin{align}\label{2.6.2}
V_{\la}\otimes V_{\omega_1}=&V_{\la+\omega_1}\oplus V_{\la+\omega_1-\alpha_1}\oplus \dots \oplus V_{\la+\omega_1-\alpha_1-\dots-\alpha_r} \nonumber\\&\oplus V_{\la+\omega_1-\alpha_1-\dots-\alpha_{r-2}-2\alpha_{r-1}-\alpha_r}\oplus\dots\oplus V_{\la+\omega_1-2\alpha_1-\dots-2\alpha_{r-1}-\alpha_r}\nonumber\\
=&V_{(\la_1+1,\la_2,\dots,\la_r)}\oplus V_{(\la_1-1,\la_2+1,\la_3,\dots,\la_r)}\oplus
V_{(\la_1,\dots,\la_{k-1},\la_{k}-1,\la_{k+1}+1,\dots,\la_r)}\nonumber\\ &\oplus\dots\oplus V_{(\la_1,\la_2,\dots,\la_{r-1}-1,\la_r+1)}\oplus V_{(\la_1,\la_2,\dots,\la_{r-2},\la_{r-1}+1,\la_r-1)}\nonumber\\
&\oplus V_{(\la_1,\la_2,\dots,\la_{r-3},\la_{r-2}+1,\la_{r-1}-1,\la_r)}\oplus \dots \oplus V_{(\la_1+1,\la_2-1,\la_3,\dots,\la_r)}\oplus V_{(\la_1-1,\la_2,\dots,\la_r)},
\end{align}
with the convention that the summands with non-dominant highest weights are omitted. Note, in particular, all multiplicities are 1.

We call an $r$-tuple of integers $\bs l=(l_1,\dots,l_r)$ {\it admissible} if the $V_{\la+\omega_1-\alpha(\bs l)}$ appears in \eqref{2.6.2}.

The admissible $r$-tuples $\bs l$ have the form
\beq\label{eq adm C}(\underbrace{1,\dots,1}_{k_1 \text{ ones}},0,\dots,0)\text{ or } (\underbrace{1,\dots,1}_{k_2 \text{ ones}},2,\dots,2,1),
\eeq
where $k_1=0,1,\dots,r$ and $k_2=0,1,\dots,r-2$. In the first case the length $l=l_1+\dots +l_r$ is $k_1$ and in the second case $2r-k_2-1$. It follows that
different admissible $r$-tuples have different length and, therefore, admissible tuples $\bs l$ are parametrized by length $l\in \{0,1,\dots,2r-1\}$. We call a nonnegative integer $l$ {\it admissible} if it is the length of an admissible  $r$-tuple $\bs l$. More precisely, a nonnegative integer $l$ is admissible if $l=0$ or if $l\lle r$, $\la_l>0$ or if $r< l\lle 2r-1$, $\la_{2r-l}>0$.

Similarly to the case of type ${\rm B}_r$, see Theorem \ref{thm B2 generic} and  Corollary \ref{thm sol B}, we obtain the solutions to the Bethe ansatz equations for $V_\la\otimes V_{\omega_1}$.
\begin{thm}\label{thm sol C}
Let $\g=\mathfrak{sp}(2r)$. Let $\bs l$ be as in \eqref{eq adm C}. If $\bs l$ is not admissible then the Bethe ansatz equation \eqref{eq:bae} associated to $\bm \La,\bm z,\bm l$ has no solutions. If $\bs l$ is admissible then the Bethe ansatz equation \eqref{eq:bae} associated to $\bm \La,\bm z,\bm l$ has exactly one solution represented by the following $r$-tuple of polynomials $\bm y^{(l)}$.

For $l=0,1,\dots,r-1$, we have $\bm y^{(l)}=(x-c_1^{(l)},\dots,x-c_l^{(l)},1,\dots,1)$, where $c_j^{(l)}$ are given by \eqref{3.4.12}.

For $l=r$, we have $\bm y^{(l)}=(x-c_1^{(r)},\dots,x-c_r^{(r)})$, where
$$
c_j^{(r)}=\prod_{i=1}^{j}\frac{\la_i+\dots+\la_{r-1}+2\la_r+r+1-i}{\la_i+\dots+\la_{r-1}+2\la_r+r+2-i},\quad j=1,\dots,r-1,
$$
$$
c_r^{(r)}=\frac{\la_r}{\la_r+1}\prod_{i=1}^{r-1}\frac{\la_i+\dots+\la_{r-1}+2\la_r+r+1-i}{\la_i+\dots+\la_{r-1}+2\la_r+r+2-i}.
$$

For $l=r+1,\dots,2r-1$, we have $\bm y^{(l)}=(x-c_1^{(l)},\dots,x-c_{2r-l-1}^{(l)},(x-a_{2r-l}^{(l)})(x-b_{2r-l}^{(l)}),\dots,(x-a_{r-1}^{(l)})(x-b_{r-1}^{(l)}),x-c_{r}^{(l)})$, where
$$
c_j^{(l)}=\prod_{i=1}^{j}\frac{\la_i+\dots+\la_{2r-1-l}+2\la_{2r-l}+\dots+2\la_r+l+1-i}{\la_i+\dots+\la_{2r-1-l}+2\la_{2r-l}+\dots+2\la_r+l+2-i},\quad j=1,\dots,2r-l-1,
$$
\begin{align*}
c_r^{(l)}=&\prod_{i=1}^{2r-l-1}\frac{\la_i+\dots+\la_{2r-1-l}+2\la_{2r-l}+\dots+2\la_r+l+1-i}{\la_i+\dots+\la_{2r-1-l}+2\la_{2r-l}+\dots+2\la_r+l+2-i}\nonumber\\
&\times\prod_{i=2r-l}^r\frac{\la_{2r-l}+\dots+\la_i+2\la_{i+1}+\dots+2\la_r+l-i}{\la_{2r-l}+\dots+\la_i+2\la_{i+1}+\dots+2\la_r+l-i+1},
\end{align*}
\begin{align*}
a_k^{(l)}b_k^{(l)}=&\left(\prod_{i=1}^{2r-l-1}\frac{\la_i+\dots+\la_{2r-1-l}+2\la_{2r-l}+\dots+2\la_r+l+1-i}{\la_i+\dots+\la_{2r-1-l}+2\la_{2r-l}+\dots+2\la_r+l+2-i}\right)^2\nonumber\\
&\times\prod_{i=2r-l}^{k-1}\frac{\la_{2r-l}+\dots+\la_i+2\la_{i+1}+\dots+2\la_r+l-i}{\la_{2r-l}+\dots+\la_i+2\la_{i+1}+\dots+2\la_r+l-i+1}\nonumber\\
&\times\prod_{i=2r-l}^{r}\frac{\la_{2r-l}+\dots+\la_i+2\la_{i+1}+\dots+2\la_r+l-i}{\la_{2r-l}+\dots+\la_i+2\la_{i+1}+\dots+2\la_r+l-i+1}\nonumber\\
&\times\prod_{i=1}^{r+1-k}\frac{\la_{2r-l}+\dots+\la_{r+1-i}+l+1-i-r}{\la_{2r-l}+\dots+\la_{r+1-i}+l+2-i-r}
\end{align*}
and
\begin{align*}
a_k^{(l)}+b_k^{(l)}=&\prod_{i=1}^{2r-l-1}\frac{\la_i+\dots+\la_{2r-1-l}+2\la_{2r-l}+\dots+2\la_r+l+1-i}{\la_i+\dots+\la_{2r-1-l}+2\la_{2r-l}+\dots+2\la_r+l+2-i}\nonumber\\
&\times\prod_{i=2r-l}^{k-1}\frac{\la_{2r-l}+\dots+\la_i+2\la_{i+1}+\dots+2\la_r+l-i}{\la_{2r-l}+\dots+\la_i+2\la_{i+1}+\dots+2\la_r+l-i+1}\nonumber\\
&\times\left(\frac{2\la_{2r-l}+\dots+2\la_r+2l-2r}{2\la_{2r-l}+\dots+2\la_r+2l+1-2r}+\frac{2\la_{2r-l}+\dots+2\la_r+2l+2-2r}{2\la_{2r-l}+\dots+2\la_r+2l+1-2r}\right.\nonumber\\
&\times \prod_{i=k}^{r}\frac{\la_{2r-l}+\dots+\la_i+2\la_{i+1}+\dots+2\la_r+l-i}{\la_{2r-l}+\dots+\la_i+2\la_{i+1}+\dots+2\la_r+l-i+1}\nonumber\\
&\times \left.\prod_{i=k}^{r}\frac{\la_{2r-l}+\dots+\la_{i}+l+i-2r}{\la_{2r-l}+\dots+\la_{i}+l+i+1-2r}\right),
\end{align*}
for $k=2r-l,\dots,r-1$.$\hfill\square$
\end{thm}
Therefore, in parallel to Theorem \ref{thm B generic}, we have the completeness of Bethe ansatz.
\begin{thm}
Let $\g=\mathfrak{sp}(2r)$ and $\la\in\cP^+$. For a generic $(N+1)$-tuple of distinct complex numbers $\bm z=(z_0,z_1,\dots,z_N)$, the Gaudin Hamiltonians $(\cH_0,\cH_1,\dots,\cH_N)$ acting in ${\rm Sing}\left(V_\la\otimes V_{\omega_1}^{\otimes N}\right)$ are diagonalizable and have simple joint spectrum. Moreover, for generic $\bm z$ there exists a set of solutions $\{\bm t_i,~i\in I\}$ of the Bethe ansatz equation \eqref{eq:bae} such that the corresponding Bethe vectors $\{\omega(\bm z,\bm t_i),~i\in I\}$ form a basis of ${\rm Sing}\left(V_\la\otimes V_{\omega_1}^{\otimes N}\right)$.$\hfill\square$
\end{thm}

Similarly to Section \ref{sec oper}, following \cite{MV2}, we introduce a linear differential operator $D(\bm y)$ of order $2r+1$ by the formula
\begin{align*}
D_\la(\bm y)=&\left(\pa-\ln'\left(\frac{T_1^2\dots T_{r-1}^2 T_r^2}{y_1}\right)\right)\left(\pa-\ln'\left(\frac{y_1T_1^2\dots T_{r-1}^2T_r^2}{y_2T_1}\right)\right)\dots\\
&\times \left(\pa-\ln'\left(\frac{y_{r-2}T_1^2\dots T_{r-1}^2 T_r^2}{y_{r-1}T_1\dots T_{r-2}}\right)\right)\dots \left(\pa-\ln'\left(\frac{y_{r-1}T_1^2\dots T_{r-1}^2 T_r^2}{y_r^2 T_1\dots T_{r-1}}\right)\right)\\
&\times \left(\pa-\ln'\left(T_1\dots T_{r}\right)\right)\left(\pa-\ln'\left(\frac{y_{r}^2T_1\dots T_{r-1} }{y_{r-1}}\right)\right)\\
&\times \left(\pa-\ln'\left(\frac{y_{r-1}T_1\dots T_{r-2}}{y_{r-2}}\right)\right)\dots\left(\pa-\ln'\left(\frac{y_2T_1}{y_1}\right)\right)(\pa-\ln'(y_1)),
\end{align*}
where $T_i$, $i=1,\dots,r$, are given by \eqref{eq T}.

If $\bs y$ is an $r$-tuple of polynomials representing a critical point associated with integral dominant weights $\La_1,\dots,\La_n$ and points $z_1,\dots,z_n$ of type $\mathrm{C}_r$, then by \cite{MV2}, the kernel of $D_\la(\bm y)$ is a self-dual space of polynomials. By \cite{MM} the coefficients of $D_\la(\bm y)$ are eigenvalues of higher Gaudin Hamiltonians acting on the Bethe vector related to $\bs y$.

For admissible $l$ and $\la \in\mathfrak{h}^*$, define $a_\la^l(1),\dots, a_\la^l(r)$ as follows.

For $l=0,\dots,r-1$, $i=1,\dots, l$, set $a_\la^l(i)=\la_i+\dots+\la_l+l+1-i$. For $l=0,\dots,r$, $i=l+1, \dots, r$, set $a_\la^l(i)=0$.

For $l=r$, $i=1,\dots, r-1$, set $a_\la^l(i)=\la_i+\dots+\la_{r-1}+2\la_r+r+2-i$ and $a_\la^r(r)=\la_r+1$.

For $l=r+1,\dots,2r-1$, set $k=2r-l-1$. Then for $i=1,\dots, k$, set $$a_\la^l(i)=\la_i+\dots+\la_k+2\la_{k+1}+\dots+2\la_{r-1}+2\la_r+2r+1-k-i$$
and for $i=k+1,\dots, r-1$, set $a_\la^l(i)=2\la_{k+1}+\dots+2\la_{r-1}+2\la_r+2r-2k$ and $a_\la^l(r)=\la_{k+1}+\dots+\la_{r-1}+\la_r+r-k$.

\begin{prop}\label{diff C} Let the r-tuple $\bm y$ represent the solution of the Bethe ansatz equation \eqref{eq:bae} associated to $\bm\La$, $\bm z$ and admissible $\bs l$, where $\la\in\cP^+$. Then $D_\la(\bm y)=D_\la(x^{a_\la^l(1)},\dots,x^{a_\la^l(r)})$.
\qed
\end{prop}

\subsection{The case of \texorpdfstring{${\rm D}_r$}{Lg} }\label{Lg} Let $\g=\mathfrak{so}(2r)$ be the simple Lie algebra of type ${\rm D}_r$, where $r\gge 4$. We have $(\alpha_i,\alpha_i)=2$, $i=1,\dots,r$, $(\alpha_i,\alpha_{i-1})=1$, $i=1,\dots,r-1$, and  $(\alpha_r,\alpha_{r-2})=1$, $(\alpha_r,\alpha_{r-1})=0$. We work with data $\bm \La=(\la,\omega_1)$, $\bm z=(0,1)$, where $\la\in\mathcal P^+$.

We have
\begin{align}\label{eq dec D}
V_{\la}\otimes V_{\omega_1}=&V_{\la+\omega_1}\oplus V_{\la+\omega_1-\alpha_1}\oplus \dots \oplus V_{\la+\omega_1-\alpha_1-\dots-\alpha_{r}}\oplus V_{\la+\omega_1-\alpha_1-\dots-\alpha_{r-2}-\alpha_r} \nonumber\\&\oplus  V_{\la+\omega_1-\alpha_1-\dots-\alpha_{r-3}-2\alpha_{r-2}-\alpha_{r-1}-\alpha_r}\oplus\dots\oplus V_{\la+\omega_1-2\alpha_1-\dots-2\alpha_{r-2}-\alpha_{r-1}-\alpha_r}\nonumber\\
=&V_{(\la_1+1,\la_2,\dots,\la_r)}\oplus V_{(\la_1-1,\la_2+1,\la_3,\dots,\la_r)}\oplus\dots\oplus
V_{(\la_1,\dots,\la_{k-1},\la_{k}-1,\la_{k+1}+1,\la_{k+2},\dots,\la_r)}\nonumber\\ &\oplus\dots\oplus V_{(\la_1,\la_2,\dots,\la_{r-2}-1,\la_{r-1}+1,\la_r+1)}\oplus V_{(\la_1,\la_2,\dots,\la_{r-2},\la_{r-1}-1,\la_r+1)}\nonumber\\
&\oplus V_{(\la_1,\la_2,\dots,\la_{r-2}+1,\la_{r-1}-1,\la_r-1)}\oplus V_{(\la_1,\la_2,\dots,\la_{r-2},\la_{r-1}+1,\la_r-1)}\nonumber\\& \oplus V_{(\la_1,\la_2,\dots,\la_{r-4},\la_{r-3}+1,\la_{r-2}-1,\la_{r-1},\la_r)}\oplus \dots\oplus V_{(\la_1,\dots,\la_{k-2},\la_{k-1}+1,\la_{k}-1,\la_{k+1},\dots,\la_r)}\nonumber\\ &\oplus\dots \oplus V_{(\la_1+1,\la_2-1,\la_3,\dots,\la_r)}\oplus V_{(\la_1-1,\la_2,\dots,\la_r)},
\end{align}
with the convention that the summands with non-dominant highest weights are omitted. Note, in particular, all multiplicities are 1.

We call an $r$-tuple of integers $\bm l=(l_1,\dots,l_r)$ \emph{admissible} if the $V_{\la+\omega_1-\alpha(\bm l)}$ appears in (\ref{eq dec D}).

The admissible $r$-tuple $\bm l$ have the form
\beq\label{adm l D}
(\underbrace{1,\dots,1}_{k_1 \text{ ones}},0,\dots,0) \text{ or } (\underbrace{1,\dots,1}_{r-2 \text{ ones}},1,0)\text{ or } (\underbrace{1,\dots,1}_{r-2 \text{ ones}},0,1) \text{ or } (\underbrace{1,\dots,1}_{k_2 \text{ ones}},2,\dots,2,1,1),
\eeq
where $k_1=0,\dots,r-2,r$ and $k_2=0,\dots,r-2$. In the first case the length $l=l_1+\dots+l_r$ is $k_1$, in the second and third cases $r-1$ and in the forth case $2r-k_2-2$. It follows that different admissible $r$-tuples in the first and forth cases have different length and, therefore, admissible tuples $\bm l$ of these types are parametrized by length $l\in\{0,1,\dots,r-2,r,\dots,2r-2\}$. We denote the lengths in the second and third cases by $r-1$ and $\overline{r-1}$, respectively. More precisely, for $l\in\{0,1,\dots,r-1,\overline{r-1},r,\dots,2r-2\}$, $l$ is a length of an admissible $r$-tuple $\bm l$ if $l=0$ or $l\lle r-1$, $\la_l>0$ or if $l=\overline{r-1}$, $\la_r>0$ or if $l=r$, $\la_{r-1}>0$ and $\la_r>0$ or if $l\gge r+1$, $\la_{2r-l-1}>0$. We call such $l$ {\it admissible}.

Similarly to the case of type ${\rm B}_r$, see Theorem \ref{thm B2 generic} and  Corollary \ref{thm sol B}, we obtain the solutions to Bethe ansatz equations for $V_\la\otimes V_{\omega_1}$.

\begin{thm}\label{thm sol D}
Let $\g=\mathfrak{so}(2r)$. Let $\bs l$ be as in \eqref{adm l D}. If $\bs l$ is not admissible then the Bethe ansatz equation \eqref{eq:bae} associated to $\bm \La,\bm z,\bs l$ has no solutions. If $\bs l$ is admissible then the Bethe ansatz equation \eqref{eq:bae} associated to $\bm \La,\bm z,\bs l$ has exactly one solution represented by the following $r$-tuple of polynomials $\bm y^{(l)}$.

For $l=0,1,\dots, r-1$, we have $\bm y^{(l)}=(x-c_1^{(l)},\dots,x-c_l^{(l)},1,\dots,1)$, where $c_j^{(l)}$ are given by \eqref{3.4.12}.

For $l=\overline{r-1}$, we have $\bm y^{(\overline{r-1})}=(x-c_1^{(\overline{r-1})},\dots,x-c_{r-2}^{(\overline{r-1})},1,x-c_{r}^{(\overline{r-1})})$, where
$$
c_j^{(\overline{r-1})}=\prod_{i=1}^j\frac{\la_i+\dots+\la_{r-2}+\la_r+r-1-i}{\la_i+\dots+\la_{r-2}+\la_r+r-i},\quad j=1,\dots,r-2,
$$
and
$$
c_r^{(\overline{r-1})}=\frac{\la_r}{\la_r+1}\prod_{i=1}^{r-2}\frac{\la_i+\dots+\la_{r-2}+\la_r+r-1-i}{\la_i+\dots+\la_{r-2}+\la_r+r-i}.
$$

For $l=r$, we have $\bm y^{(r)}=(x-c_1^{(r)},\dots,x-c_{r}^{(r)})$, where
\[
c_j^{(r)}=\prod_{i=1}^{j}\frac{\la_i+\dots+\la_{r}+r-i}{\la_i+\dots+\la_{r}+r+1-i},\quad j=1,\dots,r-2,
\]
\[
c_{r-1}^{(r)}=\frac{\la_{r-1}}{\la_{r-1}+1}\prod_{i=1}^{r-2}\frac{\la_i+\dots+\la_{r}+r-i}{\la_i+\dots+\la_{r}+r+1-i},
\]
and
\[
c_{r}^{(r)}=\frac{\la_{r}}{\la_{r}+1}\prod_{i=1}^{r-2}\frac{\la_i+\dots+\la_{r}+r-i}{\la_i+\dots+\la_{r}+r+1-i}.
\]

For $l=r+1,\dots,2r-2$, we have $$\bm y^{(l)}=(x-c_1^{(l)},\dots,x-c_{2r-l-2}^{(l)},(x-a_{2r-l-1}^{(l)})(x-b_{2r-l-1}^{(l)}),\dots,(x-a_{r-2}^{(l)})(x-b_{r-2}^{(l)}),x-c_{r-1}^{(l)},x-c_r^{(l)}),$$ where
\begin{align*}
c_j^{(l)}=&\prod_{i=1}^{j}\frac{\la_i+\dots+\la_{2r-2-l}+2\la_{2r-l-1}+\dots+2\la_{r-2}+\la_{r-1}+\la_r+l-i}{\la_i+\dots+\la_{2r-2-l}+2\la_{2r-l-1}+\dots+2\la_{r-2}+\la_{r-1}+\la_r+l+1-i},
\end{align*}
$j=1,\dots,2r-l-2$,
\begin{align*}
c_{r-1}^{(l)}=&\prod_{i=1}^{2r-2-l}\frac{\la_i+\dots+\la_{2r-2-l}+2\la_{2r-l-1}+\dots+2\la_{r-2}+\la_{r-1}+\la_r+l-i}{\la_i+\dots+\la_{2r-2-l}+2\la_{2r-l-1}+\dots+2\la_{r-2}+\la_{r-1}+\la_r+l+1-i}\nonumber\\
&\times \prod_{i=2r-1-l}^{r-2}\frac{\la_{2r-1-l}+\dots+\la_i+2\la_{i+1}+\dots+2\la_{r-2}+\la_{r-1}+\la_r+l-i-1}{\la_{2r-1-l}+\dots+\la_i+2\la_{i+1}+\dots+2\la_{r-2}+\la_{r-1}+\la_r+l-i}\nonumber \\
&\times  \frac{\la_{2r-1-l}+\dots+\la_{r-2}+\la_{r-1}+l-r}{\la_{2r-1-l}+\dots+\la_{r-2}+\la_{r-1}+l-r+1},
\end{align*}
\begin{align*}
c_{r}^{(l)}=&\prod_{i=1}^{2r-2-l}\frac{\la_i+\dots+\la_{2r-2-l}+2\la_{2r-l-1}+\dots+2\la_{r-2}+\la_{r-1}+\la_r+l-i}{\la_i+\dots+\la_{2r-2-l}+2\la_{2r-l-1}+\dots+2\la_{r-2}+\la_{r-1}+\la_r+l+1-i}\nonumber\\
& \times \prod_{i=2r-1-l}^{r-2}\frac{\la_{2r-1-l}+\dots+\la_i+2\la_{i+1}+\dots+2\la_{r-2}+\la_{r-1}+\la_r+l-i-1}{\la_{2r-1-l}+\dots+\la_i+2\la_{i+1}+\dots+2\la_{r-2}+\la_{r-1}+\la_r+l-i}\nonumber \\
&\times \frac{\la_{2r-1-l}+\dots+\la_{r-2}+\la_{r}+l-r}{\la_{2r-1-l}+\dots+\la_{r-2}+\la_{r}+l-r+1},
\end{align*}

\begin{align*}
a_k^{(l)}b_k^{(l)}=&\left(\prod_{i=1}^{2r-l-2}\frac{\la_i+\dots+\la_{2r-2-l}+2\la_{2r-l-1}+\dots+2\la_{r-2}+\la_{r-1}+\la_r+l-i}{\la_i+\dots+\la_{2r-2-l}+2\la_{2r-l-1}+\dots+2\la_{r-2}+\la_{r-1}+\la_r+l-i+1}\right)^2\nonumber\\
&\times\prod_{i=2r-1-l}^{r-2}\frac{\la_{2r-1-l}+\dots+\la_i+2\la_{i+1}+\dots+2\la_{r-2}+\la_{r-1}+\la_r+l-i-1}{\la_{2r-1-l}+\dots+\la_i+2\la_{i+1}+\dots+2\la_{r-2}+\la_{r-1}+\la_r+l-i}\nonumber\\
&\times\prod_{i=2r-1-l}^{k-1}\frac{\la_{2r-1-l}+\dots+\la_i+2\la_{i+1}+\dots+2\la_{r-2}+\la_{r-1}+\la_r+l-i-1}{\la_{2r-1-l}+\dots+\la_i+2\la_{i+1}+\dots+2\la_{r-2}+\la_{r-1}+\la_r+l-i}\nonumber\\&\times
\prod_{i=k}^{r-2}\frac{\la_{2r-l-1}+\dots+\la_i+l+i+1-2r}{\la_{2r-l-1}+\dots+\la_i+l+i+2-2r}\nonumber\\
&\times\frac{\la_{2r-1-l}+\dots+\la_{r-2}+\la_{r-1}+l-r}{\la_{2r-1-l}+\dots+\la_{r-2}+\la_{r-1}+l-r+1}\cdot\frac{\la_{2r-1-l}+\dots+\la_{r-2}+\la_{r}+l-r}{\la_{2r-1-l}+\dots+\la_{r-2}+\la_{r}+l-r+1},
\end{align*}
and
\begin{align*}
a_k^{(l)}+b_k^{(l)}=&\prod_{i=1}^{2r-l-2}\frac{\la_i+\dots+\la_{2r-2-l}+2\la_{2r-l-1}+\dots+2\la_{r-2}+\la_{r-1}+\la_r+l-i}{\la_i+\dots+\la_{2r-2-l}+2\la_{2r-l-1}+\dots+2\la_{r-2}+\la_{r-1}+\la_r+l-i+1}\nonumber\\
&\times\prod_{i=2r-1-l}^{k-1}\frac{\la_{2r-1-l}+\dots+\la_i+2\la_{i+1}+\dots+2\la_{r-2}+\la_{r-1}+\la_r+l-i-1}{\la_{2r-1-l}+\dots+\la_i+2\la_{i+1}+\dots+2\la_{r-2}+\la_{r-1}+\la_r+l-i}\nonumber\\
&\times\left(\frac{2\la_{2r-l-1}+\dots+2\la_{r-2}+\la_{r-1}+\la_r+2l-2r}{2\la_{2r-l-1}+\dots+2\la_{r-2}+\la_{r-1}+\la_r+2l-2r+1}\right.\nonumber\\&+\frac{2\la_{2r-l-1}+\dots+2\la_{r-2}+\la_{r-1}+\la_r+2l-2r+2}{2\la_{2r-l-1}+\dots+2\la_{r-2}+\la_{r-1}+\la_r+2l-2r+1}\nonumber\\
&\times \prod_{i=k}^{r-2}\frac{\la_{2r-1-l}+\dots+\la_i+2\la_{i+1}+\dots+2\la_{r-2}+\la_{r-1}+\la_r+l-i-1}{\la_{2r-1-l}+\dots+\la_i+2\la_{i+1}+\dots+2\la_{r-2}+\la_{r-1}+\la_r+l-i}\nonumber\\
&\times \frac{\la_{2r-1-l}+\dots+\la_{r-2}+\la_{r-1}+l-r}{\la_{2r-1-l}+\dots+\la_{r-2}+\la_{r-1}+l-r+1}\cdot\frac{\la_{2r-1-l}+\dots+\la_{r-2}+\la_{r}+l-r}{\la_{2r-1-l}+\dots+\la_{r-2}+\la_{r}+l-r+1}\nonumber\\
&\times \left.\prod_{i=k}^{r-2}\frac{\la_{2r-l-1}+\dots+\la_i+l+i+1-2r}{\la_{2r-l-1}+\dots+\la_i+l+i+2-2r}\right),
\end{align*}
$k=2r-1-l,\dots,r-2$.$\hfill\square$
\end{thm}
Note that the formulas above with $r=3$ correspond to solutions of the Bethe ansatz equations of type ${\rm A}_3$ and $\bm \La=(\la,\omega_2)$. These formulas were given in Theorem 5.5, \cite{MV1}.

Then we deduce the analog of Theorem \ref{thm B generic}.

\begin{thm}
Let $\g=\mathfrak{so}(2r)$ and $\la\in\cP^+$. For a generic $(N+1)$-tuple of distinct complex numbers $\bm z=(z_0,z_1,\dots,z_N)$, the Gaudin Hamiltonians $(\cH_0,\cH_1,\dots,\cH_N)$ acting in ${\rm Sing}\left(V_\la\otimes V_{\omega_1}^{\otimes N}\right)$ are diagonalizable. Moreover, for generic $\bm z$ there exists a set of solutions $\{\bm t_i,~i\in I\}$ of the Bethe ansatz equation \eqref{eq:bae} such that the corresponding Bethe vectors $\{\omega(\bm z,\bm t_i),~i\in I\}$ form a basis of ${\rm Sing}\left(V_\la\otimes V_{\omega_1}^{\otimes N}\right)$.$\hfill\square$
\end{thm}
For type D, the algebra has a non-trivial diagram automorphism which leads to degeneracy of the spectrum.
 For example, if $\la_{r-1}=\la_{r}$, then the Bethe vectors corresponding to the critical points $\bm y^{(r-1)}$ and ${\bm y}^{(\overline{r-1})}$ are eigenvectors of the Gaudin Hamiltonian $\cH:=\cH_1=-\cH_2$ with the same eigenvalue. In particular Proposition \ref{prop3.18} is not applicable since the two corresponding
summands in \eqref{eq dec D} have non-comparable highest weights.

\end{document}